\documentclass{amsart}

\usepackage{amsfonts}
\usepackage{amsmath}
\usepackage{amsthm}
\usepackage[utf8]{inputenc}
\usepackage[english]{babel}
\usepackage{epstopdf}
\usepackage{bmpsize}
\usepackage{epsfig}
\usepackage[english]{varioref}
\usepackage{amssymb}
\usepackage{multicol}
\usepackage{dcolumn}
\usepackage{geometry}
\usepackage{fancyhdr}
\usepackage[mathcal]{eucal}
\usepackage{mathrsfs}
\usepackage{color, colortbl}
\usepackage{microtype}
\usepackage{longtable}
\usepackage[toc,page]{appendix}
\usepackage{bm}
\usepackage{pifont}
\usepackage{fleqn}
\usepackage{graphicx}
\usepackage{txfonts}
\usepackage{paralist}
\usepackage{graphics} 
\usepackage{graphicx}  
\usepackage[shortcuts]{extdash}
\usepackage{epstopdf}
\usepackage[colorlinks=true]{hyperref}
\hypersetup{urlcolor=blue, citecolor=red}

\newtheorem{theorem}{Theorem}[section]

\newtheorem{lemma}{Lemma}[section]
\newtheorem{proposition}{Proposition}[section]

\theoremstyle{definition}
\newtheorem{definition}{Definition}[section]
\newtheorem{remark}{Remark}[section]

\allowdisplaybreaks
\DisableLigatures{encoding = *, family = * }

\numberwithin{equation}{section}
\newcommand{\norm}[2]{{\left\|#1\right\|}_{#2}}

\newcommand{\cnk}{C_{\nu,k}}
\newcommand{\jnk}{j_{\nu,k}}
\newcommand{\jk}{j_{0,k}}
\newcommand{\NN}{\mathbb{N}}

\newcommand{\RR}{\mathbb{R}}
\newcommand{\CC}{\mathbb{C}}

\title[Boundary controllability one-dimensional singular heat equation]{Boundary controllability for a one-dimensional heat equation with a singular inverse-square potential}

\author{Umberto Biccari}
\address{DeustoTech, University of Deusto, 48007 Bilbao, Basque Country, Spain.}
\address{Facultad de Ingenier\'{\i}a, Universidad de Deusto, Avda Universidades 24, 48007 Bilbao, Basque Country, Spain.}
\email{umberto.biccari@deusto.es, u.biccari@gmail.com}

\thanks{
This project has received funding from the European Research Council (ERC) under the European Union's Horizon 2020 research and innovation programme (grant agreement No. 694126-DyCon). Moreover, this work was partially supported by the Grants MTM2014-52347, MTM2017-92996 of MINECO (Spain), and AFOSR Grant FA9550-18-1-0242.}

\begin{document}

\bibliographystyle{acm}

\begin{abstract}
We analyze controllability properties for the one-dimensional heat equation with singular inverse-square potential
\begin{align*}
	u_t-u_{xx}-\frac{\mu}{x^2}u=0,\;\;\; (x,t)\in(0,1)\times(0,T).
\end{align*}
For any $\mu<1/4$, we prove that the equation is null controllable through a boundary control $f\in H^1(0,T)$ acting at the singularity point $x=0$. This result is obtained employing the moment method by Fattorini and Russell.
\end{abstract}

\maketitle

\section{Introduction and main results}\label{intro}
Let $T>0$ and set $Q:=(0,1)\times(0,T)$. The aim of this work is to prove boundary controllability for a one dimensional heat equation with a singular inverse-square potential. In particular we are interested in the case in which the potential arises at the boundary and the  control is located on the singularity point. In other words, given the operator 
\begin{align}\label{singular_op}
	A_{\mu}:=-d_x^2-\frac{\mu}{x^2}I,\;\;\;\mu\in\RR,
\end{align} 
we are going to consider the heat equation
\begin{align}\label{heat_hardy_nhb}
	\begin{cases}
		\displaystyle u_t-u_{xx}-\frac{\mu}{x^2}u=0, & (x,t)\in Q
		\\[5pt] 
		x^{-\alpha}u(x,t)\big|_{x=0}=f(t),\;u(1,t)=0, & t\in (0,T)
		\\[5pt] 
		u(x,0)=u_0(x), & x\in (0,1)
	\end{cases}
\end{align}
with the intent of proving that it is possible to find a control function $f$ in an appropriate functional space $X$ such that solutions of \eqref{heat_hardy_nhb} corresponding to initial data $u_0\in L^2(0,1)$ satisfy
\begin{align}\label{control_def}
	u(x,T)=u_T(x).
\end{align}

A first important aspect that we want to underline is the non standard formulation of the boundary conditions in \eqref{heat_hardy_nhb}. Indeed, due to the presence of the singularity at $x=0$ it turns out that it is not possible to impose a boundary condition of the type $u(0,t)=f(t)\neq 0$. Instead, we need to introduce the weighted boundary condition 
\begin{align}\label{1d_weighted_bc}
	\left. x^{-\alpha}u(x,t)\right|_{x=0}=f(t),
\end{align}  
where, for all $\mu\leq 1/4$, the coefficient $\alpha$ is given by
\begin{align}\label{alpha}
	\alpha=\alpha(\mu):=\frac{1}{2}\left(1-\sqrt{1-4\mu}\;\right).
\end{align}

This fact is justified by the observation that the general solution of the second order elliptic equation $u_{xx}+(\,\mu/x^2)u=0$ may be calculated explicitly and it is given by
\begin{align*}
	u(x) = C_1x^{\,\frac{1}{2}-\frac{1}{2}\sqrt{1-4\mu}} + C_2x^{\,\frac{1}{2}+\frac{1}{2}\sqrt{1-4\mu}},
\end{align*}
with $(C_1,C_2)\neq(0,0)$. Therefore,
\begin{align}\label{1d_solution_trace}
	\begin{cases}
		u(0)=0, & \textrm{ for } \mu > 0,
		\\
		u(0)=\pm\infty, & \textrm{ for } \mu < 0,
	\end{cases}
\end{align}
where the sign of $u(0)$ for $\mu<0$ is given by the sign of the constant $C_1$. On the other hand, we have
\begin{align*}
	\lim_{x\to 0^+}x^{\,-\frac{1}{2}+\frac{1}{2}\sqrt{1-4\mu}}u(x) = \lim_{x\to 0^+}x^{-\alpha}u(x) = \left\{\begin{array}{ll} 
C_1, & \textrm{ for } \mu <1/4,
	\\
	C_1+C_2, & \textrm{ for } \mu =1/4.
	\end{array}\right.
\end{align*}

We remark that in \eqref{1d_solution_trace} we are not considering the case $\mu=0$. This case, indeed, corresponds simply to a one-dimensional Laplace equation for which, of course, we do not need any further analysis. Moreover, we notice that for $\mu=0$ we have also $\alpha=0$ and, therefore, the boundary condition \eqref{1d_weighted_bc} becomes $u(0,t)=f(t)$, which is consistent with the classical theory. Finally, it is evident from the argument above that $x^{-\alpha}$ is the sharp weight for defining a non-homogeneous boundary condition at $x=0$.  As we shall see with more details later, the parameter $\alpha$ plays a fundamental role in our analysis.

As it by now well known, when dealing with equations of the type of \eqref{heat_hardy_nhb} the constant $\mu$ in the definition of the operator $A_{\mu}$ plays a crucial role. In fact, even if in principle it could assume every real value, it has been shown (see, e.g., \cite{vazquez2000hardy} and the references therein) that there is an upper bound $\mu^*$ above which this kind of equations is ill-posed. This upper bound is given by the critical constant in the Hardy inequality, guaranteeing that, for every $z\in H_0^1(0,1)$, we have $z/x\in L^2(0,1)$ and 
\begin{align}\label{hardy}
	\mu^*\int_0^1\frac{z^2}{x^2}dx\leq\int_0^1z_x^2dx.
\end{align} 

Moreover, it is also well known that, in the one-dimensional case, this critical Hardy constant takes the value  $\mu^*=1/4$. 

Singular inverse-square potentials arise in quantum cosmology (\cite{berestycki1997existence}), in electron capture problems (\cite{giri2008electron}), but also in the linearization of reaction-diffusion problems involving the heat equation with supercritical reaction term (\cite{azorero1998hardy}). Also for these reasons, starting from the pioneering work \cite{baras1984heat}, evolution problems involving this kind of potentials have been intensively studied in the last decades. 

Moreover, it is by now well known that equations of the type of \eqref{heat_hardy_nhb} are closely related, through an appropriate change of variables (see, for instance, \cite[Section 4]{martinez2006carleman}), to another class of PDE problems with variable degenerate coefficients, i.e. in the form 
\begin{align}\label{1d_heat_deg}
	u_t-(a(x)u_x)_x=0,\;\;\;\alpha\in(0,1),\;\;\;(x,t)\in Q,
\end{align}
where $a(x)$ vanishes at a certain $x_0\in[0,1]$.

In the recent past, it has been given many attention to this other kind of equations. In particular, several controllability results have been obtained, employing both distributed and boundary controls.

Concerning interior controllability, among other works, we can mention \cite{cannarsa2005null,cannarsa2008carleman,martinez2006carleman}, where the authors obtained the null-controllability for \eqref{1d_heat_deg} by means of a distributed control supported in a non-empty subset $\omega\subset (0,1)$. Moreover, it is worth to mention also the book \cite{cannarsa2016global}, which contains a complete survey of the existing results of interior controllability for \eqref{1d_heat_deg}, and where it is also treated the multi-dimensional case.

Concerning instead the boundary controllability, this problem was firstly addressed in \cite{cannarsa2012unique}, where the authors considered the case $a(x)=x^{\,\beta}$, $\beta\in (0,1)$ and proved approximate controllability acting from $x=0$.

In \cite{cannarsa2015cost,gueye2014exact}, it is considered again the case $a(x)=x^{\,\beta}$, $\beta\in (0,1)$, and it is proved the null controllability for \eqref{1d_heat_deg}, again from $x=0$. In particular, in \cite{cannarsa2015cost}, it is also presented an analysis of the cost of null controllability and of its dependence on the parameter $\beta$ and on the time horizon $T$. Finally, in \cite{cannarsa2018cost}, it is studied the case $\beta\in [1,2)$ and analogous results as in \cite{cannarsa2015cost} are obtained. 

Also for evolution equations with singular inverse-square potentials the controllability problem has already been addressed in the past. Among other works, we recall here  \cite{biccari2016null,cazacu2012schrodinger,cazacu2014controllability,ervedoza2008control,vancostenoble2008null,vancostenoble2009hardy}. 
In all these articles, the authors analyzed heat and wave equations involving a potential of the type $\mu/|x|^2$ on a bounded regular domain $\Omega\subset\RR^N$, $N\geq 3$, and proved null controllability choosing a control region inside the domain, away from the singularity point $x=0$. In particular, in \cite{cazacu2014controllability}, it is considered the special case of a singularity located on the boundary of the domain $\Omega$, while \cite{biccari2016null} treats the case of a singularity distributed all over the boundary of $\Omega$. In addition, it is worth to mention also \cite{vancostenoble2011improved}, where the author treats a problem involving degenerate coefficients and singular potentials at the same time. 

Finally, we have to mention that the boundary controllability for a one-dimensional heat equation with singular potential has recently been treated in \cite{martinez2017ccost}. Nevertheless, in that work the authors consider the case of a potential located on one extrema of the space domain (namely, $x=0$) and a control located at the other extrema ($x=1$). The main novelty of our work is that  we consider controls located at the same point in which the singularity of the potential arises. To the best of our knowledge, there are no results of controllability acting from the singularity point. The analysis of problem \eqref{heat_hardy_nhb} that we are presenting is a first step in this direction, in which the two issues just mentioned appear together. Indeed, we are going to prove that, for all $\mu<1/4$, it is possible to control the equation from the boundary, and in particular from the extrema where the singularity of the potential arises. 

Our approach will rely on the well celebrated moment method, which is by now a very classical technique for treating controllability problems for certain types of one-dimensional evolution equations. Moreover, our discussion will have many points in common with \cite{cannarsa2015cost}, where the same arguments are applied for proving the null controllability of \eqref{1d_heat_deg} acting from the degeneracy point of the coefficient $a$. 
This is not surprising since, as we mentioned, these two classes of problems are strictly related through an explicit change of variables. 

On the other hand, as we will explain with more details in Appendix \ref{appendix} at the end of this work, passing through this mentioned change of variables, we would be able to obtain a less general result. Indeed, this approach has limitations, in the sense that it is valid
under more restrictive assumptions on the coefficient $\mu$. In more detail, we would have to assume that $\mu\in [0,1/4)$, while our results are true for all $\mu<1/4$. For this reason, it is worth to analyze the controllability of \eqref{1d_heat_deg} directly, without relying on existing results for equations with degenerate coefficients. 

The paper is organized as follows. In Section \ref{results}, we state precisely the main results that we obtained. In Section \ref{well_posedness_sec}, we analyze the well-posedness for the equation (\ref{heat_hardy_nhb}), reducing it to an equation with homogeneous boundary conditions and a source term. Section \ref{spectrum_sec} is devoted to a spectral analysis of the operator $A_{\mu}$, which will be fundamental for developing the moment method. In Section \ref{moment_sec}, we present the moment problem, its solution and the proof of the controllability result. In Section \ref{cost_sec}, we analyze the cost of null controllability for our equation. Section \ref{target_reg_sec} is devoted to the proof of regularity results on the set of reachable targets and to show that the only target that is reachable for all values of the coefficient $\mu$ is $u_T=0$. Finally, in Section \ref{open_pb}, we present some open problem and perspectives related to our work.

\section{Main results}\label{results}
We recall that the main problem that we will address is the boundary controllability for equation \eqref{heat_hardy_nhb}, employing a  control located at the singularity point. In other words, given $u_0,\,u_T\in L^2(0,1)$ we wish to find a control function $f$ that drives the solution $u$ of \eqref{heat_hardy_nhb} from $u_0$ to $u_T$ in a finite time $T>0$. We will show that this controllability property is satisfied employing a control $f\in H^1(0,T)$. Moreover, we will refer to the functions $u_T$ that can be reached from $u_0$ through $f$ as \textit{reachable targets} and we will indicate their set with $\mathcal{R}_T$.

As we mentioned, this controllability result will be obtained using the moment method by Fattorini and Russel (\cite{fattorini1971exact,fattorini1974uniform}). The starting point of this method is to decompose the initial datum $u_0$ and the target $u_T$ in the basis of the eigenfunctions $(\Phi_k)_{k\geq 1}$ associated to the operator $A_{\mu}$, i.e. 
\begin{align*}
	u_0(x)=\sum_{k\geq 1} \rho_k^0\Phi_k(x),\;\;\; u_T(x)=\sum_{k\geq 1} \rho_k^T\Phi_k(x),
\end{align*}
with $(\,\rho_k^0)_{k\geq 1},\,(\,\rho_k^T)_{k\geq 1}\in\ell^2(\NN^*)$. The controllability of the equation is then reduced to an algebraic condition for the Fourier coefficients $\rho_k^0$ and $\rho_k^T$.

\noindent The main results of this paper will be the following.
\begin{enumerate}
	\item Our first concern is the null controllability of the parabolic equation \eqref{heat_hardy_nhb}. In particular, in Theorem \ref{control_thm}, for all $\mu<1/4$ we are going to prove the existence of a dense subset $\mathcal{P}_{\mu,T}\subset L^2(0,1)$ such that any $u_T\in\mathcal{P}_{\mu,T}$ is reachable with $H^1(0,T)$ controls. In other words, we will have that, for all $\mu<1/4$, $u_0\in L^2(0,1)$ and $u_T\in\mathcal{P}_{\mu,T}$, there exists a control function $f\in H^1(0,T)$ such that the corresponding solution of \eqref{heat_hardy_nhb} satisfies \eqref{control_def}. 
	
	\item Secondly, since the target $u_T=0$ will show up to be reachable for all values of the coefficient $\mu<1/4$, we will focus on the analysis of the cost of null controllability with respect to $\mu$ and, in Theorem \ref{cost_thm}, we will show that this cost blows up as $\mu\to 1/4^-$ and $T\to 0^+$.
	
	\item As a third result, we will obtain regularity properties for the targets $u_T$. In the case of the classical heat equation (i.e. when $\mu=0$) Fattorini and Russel observed in \cite{fattorini1971exact} that a reachable target is, in fact, the restriction to the interval $[0,1]$ of an analytic function. In Proposition \ref{target_reg_prop} below, we will show that a similar regularity result can be proved also in our case. 
	
	\item Finally, since the reachable set $\mathcal{R}_T$ contains a subset $\mathcal{P}_{\mu,T}$ dense in $L^2(0,1)$, we will analyse which are the targets $u_T$ that could be reached for all values of the coefficient $\mu<1/4$. In this framework, we will show in Proposition \ref{reachable_prop} that 
	\begin{align*}
		\bigcap_{\mu<1/4} \mathcal{P}_{\mu,T} = \{0\},
	\end{align*}
	meaning that the null state is the only target that we are sure that can be reached for all coefficients $\mu<1/4.$ 
\end{enumerate}

\noindent In what follows, we discuss with more details the results introduced above.

\subsection{The controllability problem}
The main result of this work will be the following.
\begin{theorem}\label{control_thm}
Let $\mu<1/4$. Consider the target function $u_T\in L^2(0,1)$ and the sequence of its Fourier coefficients $(\,\rho_k^T)_{k\geq 1}$. Then, there exists a constant $P>0$, not depending on $\mu$, such that, if
\begin{align}\label{fourier_cond_target}
	\sum_{k\geq 1}|\,\rho_k^T|k^{\frac{1}{2}-\frac{1}{2}\sqrt{1-4\mu}}e^{P\pi k} < +\infty,
\end{align}
then $u_T$ is a reachable target: given $T>0$ and $u_0\in L^2(0,1)$, there exists a control function $f\in H^1(0,T)$ such that the solution of \eqref{heat_hardy_nhb} satisfies $u(x,T)=u_T(x)$.
\end{theorem}

\begin{remark}\label{range_rem}
Notice that in the statement of Theorem \ref{control_thm} we are not allowing $\mu$ to reach the critical value $\mu^*=1/4$. Nevertheless, this assumption is justified by the fact that, as we are going to show in Section \ref{cost_sec}, as $\mu\to 1/4^-$ the cost of null controllability for \eqref{heat_hardy_nhb} blows up as $(1-4\mu)^{-1/2}$. This, of course, implies that, in the critical case, the null controllability problem would have an infinite cost and, therefore, it is not achievable.
\end{remark}

\begin{remark}
A condition of the type of \eqref{fourier_cond_target} already appears in the original works \cite{fattorini1971exact,fattorini1974uniform} of Fattorini and Russell on the moment method. 
Moreover, we notice that \eqref{fourier_cond_target} is trivially satisfied in the case o null controllability, that is, when the target function is $u_T=0$. Finally, we emphasize the fact that the constant $P$ appearing in \eqref{fourier_cond_target} does not depend on $\mu$. This is relevant in the analysis of the behaviour of the reachable set with respect to this coefficient. 
\end{remark}

\subsection{The cost of null controllability}
The second result that we were mentioning before concerns the analysis of the cost of null controllability for equation \eqref{heat_hardy_nhb}, and it will provide a justification to the impossibility of establishing Theorem \ref{control_thm} when $\mu$ is critical. 

In particular, we are going to prove that the cost of null controllability blows up as $\mu\to 1/4^-$, thus meaning that null-controllability fails in this critical case.

First of all, let us introduce the concept of the cost of controllability: given $u_0\in L^2(0,1)$, we consider the set of admissible controls that drive the solution to zero in time $T$ as
\begin{align*}
	\mathcal{U}_{\textrm{ad}}(\,\mu,u_0):=\left\{f\in H^1(0,T)\,\Big|\, \textrm{ the solution } u \textrm{ of \eqref{heat_hardy_nhb} satisfies } \eqref{control_def}\right\}.
\end{align*}
The controllability cost is then defined as
\begin{align}\label{control_cost}
	C^{H^1}(\,\mu,u_0):= \inf_{f\in\mathcal{U}_{\textrm{ad}}(\mu,u_0)}\norm{f}{H^1(0,T)}.
\end{align}
Moreover, we can also consider a global notion of the controllability cost
\begin{align}\label{control_cost_global}
	C^{H^1}_{\textrm{bd-ctr}}(\,\mu):= \sup_{\norm{u_0}{L^2(0,1)}=1}C^{H^1}(\,\mu,u_0).
\end{align}

We mention that similar notions has already being considered, for instance in \cite{fernandez2000cost}. We are going to prove the following.
\begin{theorem}\label{cost_thm}
Let $C^{H^1}(\,\mu,u_0)$ and $C^{H^1}_{\textrm{bd-ctr}}(\,\mu)$ be defined as in \eqref{control_cost} and \eqref{control_cost_global}. Then , the following assertions hold:
\begin{itemize}
	\item[(i)] Given $u_0\in L^2(0,1)$, there exists $C_1(u_0)$, independent of $\mu$, and $C$, independent of $u_0$ and $\mu$, such that 
	\begin{align}\label{control_cost_est}
		\frac{C_1(u_0)}{\sqrt{1-4\mu}}\frac{1}{\sqrt{e^{CT}-1}}\leq C^{H^1}(\,\mu,u_0) \leq \frac{Ce^{\frac{C}{T}}}{\sqrt{1-4\mu}}\norm{u_0}{L^2(0,1)}.
	\end{align}
	\item[(ii)] There exists a positive constant $C$, independent of $\mu$, such that 
	\begin{align}\label{control_cost_global_est}
		\frac{Ce^{\frac{C}{T}}}{\sqrt{1-4\mu}}\frac{1}{T(T+1)}\leq C^{H^1}_{\textrm{bd-ctr}}(\,\mu) \leq \frac{Ce^{\frac{C}{T}}}{\sqrt{1-4\mu}}.
	\end{align}	
\end{itemize}
\end{theorem}

\subsection{The regularity of the targets}

The third result of our paper is related to the regularity of the targets $u_T$ that can be reached employing a control $f\in H^1(0,T)$. We already mentioned that Fattorini and Russel in \cite{fattorini1971exact} observed that, for the classical heat equation without potential, a reachable target is the restriction to the interval $[0,1]$ of an analytic function. A similar regularity result can be obtained also in our case. In particular we will prove the following.
\begin{proposition}\label{target_reg_prop}
Consider a sequence $(\,\rho_k^T)_{k\geq 1}$ such that, for some constant $P>0$, the sequence $(\,\rho_k^Te^{Pk})_{k\geq 1}$ remains bounded. Moreover, for all $x\in[0,1]$ define 
\begin{align*}
	u_T(x)=\sum_{k\geq 1}\rho_k^T\Phi_k(x).
\end{align*}
Then, there exists a function $\tilde{F}$, holomorphic in the strip $\{z\in\CC,\,|\Im z|<P/\pi\}$, such that 
\begin{align*}
	\forall\,x\in[0,1],\;\; u_T(x)=x^{\frac{1}{2}+\frac{1}{2}\sqrt{1-4\mu}}F(x).
\end{align*}
\end{proposition}

\begin{remark}
We point out that, if the sequence of Fourier coefficients $(\,\rho_k^T)_{k\geq 1}$ is such that \eqref{fourier_cond_target} holds, this automatically implies that $(\,\rho_k^Te^{Pk})_{k\geq 1}$ remains bounded. Indeed, if the series \eqref{fourier_cond_target} is convergent, then we must have 
\begin{align}\label{sequence_lim}
	\lim_{k\to +\infty} |\,\rho_k^T|k^{\frac{1}{2}-\frac{1}{2}\sqrt{1-4\mu}}e^{Pk\pi} = 0.
\end{align}
	
\noindent Now, it is straightforward to check that 
\begin{align*}
	\lim_{k\to+\infty} k^{\frac{1}{2}-\frac{1}{2}\sqrt{1-4\mu}}= \left\{\begin{array}{ll}
	+\infty,  & \;\;\;\textrm{ for } \mu\in(0,1/4),
	\\
	1, & \;\;\;\textrm{ for } \mu=0,
	\\
	0, & \;\;\;\textrm{ for } \mu<0. 
	\end{array}\right.
\end{align*}
	
In any case, the condition \eqref{sequence_lim} can be satisfied only if the sequence $(\,\rho_k^Te^{Pk})_{k\geq 1}$ remains bounded.
\end{remark}

\subsection{Identification of the targets reachable for all $\mu<1/4$}

We will conclude this work showing that the only target which is reachable for all values of the coefficient $\mu<1/4$ is $u_T=0$. This fact will be a direct consequence of Proposition \ref{target_reg_prop}. In particular, we will prove the following.

\begin{proposition}\label{reachable_prop}
Given $\mu<1/4$, let $\mathcal{P}_{\mu,T}$ be the subset of the reachable targets $u_T$ that satisfy the condition \eqref{fourier_cond_target}. Then, we have
\begin{align*}
	\bigcap_{\mu<1/4} \mathcal{P}_{\mu,T} = \{0\}.
\end{align*}
\end{proposition}

\section{Well-posedness}\label{well_posedness_sec}
We analyze here existence and uniqueness of solutions of the heat equation \eqref{heat_hardy_nhb}. For doing that, we will follow the approach introduced in \cite{cannarsa2015cost,cannarsa2018cost}, that consists in transforming our original problem in one with Dirichlet homogeneous boundary conditions and a source term (depending on the control function $f$). To this end, let us introduce the change of variables 
\begin{align}\label{cv}
	\psi(x,t):=u(x,t)-x^{\alpha}\frac{p(x)}{p(0)}f(t),\;\;\; p(x):=1-x^{1-2\alpha}.
\end{align}

We observe that, by \eqref{alpha}, we have $1-2\alpha=\sqrt{1-4\mu}$. This means that, in what follows, we shall assume $\mu<1/4$ since, when $\mu=1/4$ the change of variables \eqref{cv} is the trivial one. Notice, however, that this assumption is not a limitation. Indeed, for critical potentials we do not expect our equation \eqref{heat_hardy_nhb} to be well posed, at least not with the boundary conditions that we are imposing. A more detailed discussion on this point will be presented in the Appendix \ref{appendix} at the end of the present work.

Now, if $u$ is a solution of \eqref{heat_hardy_nhb} then, formally, the new function $\psi$ defined in \eqref{cv} satisfies the problem
\begin{align}\label{heat_hardy_psi}
	\begin{cases}
		\displaystyle \psi_t-\psi_{xx}-\frac{\mu}{x^2}\psi=-x^{\alpha}\frac{p(x)}{p(0)}f'(t), & (x,t)\in Q
		\\[6pt] 
		\psi(0,t)=\psi(1,t)=0, & t\in (0,T)
		\\[6pt]
		\psi(x,0)=u_0(x)-x^{\alpha}\frac{p(x)}{p(0)}f(0), & x\in (0,1)
	\end{cases}
\end{align} 

Therefore, for obtaining the well-posedness of the boundary value problem \eqref{heat_hardy_nhb}, we firstly need to discuss the existence and uniqueness of solutions for heat equations of the type
\begin{align}\label{1d_heat_hardy_nh}
	\begin{cases}
		\displaystyle w_t-w_{xx}-\frac{\mu}{x^2}w=h(x,t), & (x,t)\in Q
		\\[6pt] 
		w(0,t)=w(1,t)=0, & t\in(0,T)
		\\[6pt]
		w(x,0)=w_0(x), & x\in(0,1).
	\end{cases}
\end{align}

Existence and uniqueness of solutions for problems of the type of \eqref{1d_heat_hardy_nh} is by now classical (see, for instance, \cite{vazquez2000hardy}). For the sake of completeness, in what follows we present a brief discussion on this point. 

Let us introduce the Hilbert space $H$ defined as the closure of $C_0^{\infty}(0,1)$ with respect to the norm 
\begin{align*}
	\forall w\in H_0^1(0,1),\;\;\norm{w}{H} = \left[\int_0^1 \left(w_x^2 - \frac{\mu}{x^2}w^2\right)\,dx\right]^{\frac{1}{2}}.
\end{align*}

We notice that, in view of the Hardy inequality \eqref{hardy}, this space $H$ can be defined for all $\mu\leq 1/4$. Moreover, it is simply a matter of computations to show that there exist two positive constants $M_1$ and $M_2$, depending on $\mu$, such that it holds the following inequality
\begin{align}\label{1d_norm_equiv}
	\left(1-4\mu\right)\,\int_0^1 w_x^2 + M_1\int_0^1 w^2\,dx \leq\norm{w}{H}^2\leq \left(1+4\mu\right)\,\int_0^1 w_x^2 + M_2\int_0^1 w^2\,dx.
\end{align}

It is evident that, in the sub-critical case $\mu<1/4$, from \eqref{1d_norm_equiv} it follows the identification $H=H_0^1(0,1)$ with equivalent norms. On the other hand, in the critical case $\mu=\mu^*$ this identification does not hold anymore and the space $H$ is strictly larger than $H_0^1(0,1)$. For more details on this point, we refer to \cite{vazquez2000hardy}. Notice, however, that in this work we are not interested in the case $\mu=\mu^*$.

Let us now consider the unbounded operator $\mathcal{A}:\mathcal{D}(\mathcal{A})\subset L^2(0,1)\to L^2(0,1)$, defined for all $\mu< 1/4$ as
\begin{align}\label{1d_operator_A}
	\begin{array}{c}
		\displaystyle\mathcal{D}(\mathcal{A}):=\left\{w\in H\,\Big|\, w_{xx}+\frac{\mu}{x^2}w\in L^2(0,1)\right\}, 
		\\
		\\
		\displaystyle\mathcal{A}w:=-w_{xx}-\frac{\mu}{x^2}w,
	\end{array}
\end{align} 
and whose norm is given by
\begin{align*}
	\norm{w}{\mathcal{A}} = \norm{w}{L^2(0,1)} + \norm{\mathcal{A}w}{L^2(0,1)}.
\end{align*}

With the definitions that we just gave we have that, for any $\mu< 1/4$, the operator \eqref{1d_operator_A} generates an analytic semi-group $e^{t\mathcal{A}}$ on the pivot space $L^2(0,1)$ for the equation \eqref{1d_heat_hardy_nh}. Therefore, given a source term $h\in L^2((0,1)\times (0,T))$, equation \eqref{1d_heat_hardy_nh} is well posed. 

Employing the variation of constant formula, and referring to the discussion presented in \cite{cannarsa2015cost}, we define the \textit{mild solution} $w\in C([0,T];L^2(0,1))\cap L^2(0,T;H_0^1(0,1))$ of \eqref{1d_heat_hardy_nh} as
\begin{align*}
	w(x,t)=e^{t\mathcal{A}}w_0+\int_0^t e^{(t-s)\mathcal{A}}h(x,s)\,ds.
\end{align*}
Moreover, we say that a function
\begin{align*}
	w\in C([0,T];H_0^1(0,1))\cap H^1(0,T;L^2(0,1))\cap L^2(0,T;\mathcal{D}(\mathcal{A}))
\end{align*}
is a \textit{strict solution} of \eqref{1d_heat_hardy_nh} if it satisfies the equation a.e. in $Q$ and the initial and boundary conditions for all $t\in[0,T]$ and $x\in[0,1]$. Referring to \cite{cannarsa2015cost}, we have the following.
\begin{proposition}
	If $w_0\in H_0^1(0,1)$, then the mild solution of \eqref{1d_heat_hardy_nh} is also the unique strict solution.
\end{proposition}

\noindent Of course, the above discussion for \eqref{1d_heat_hardy_nh} apply, in particular, to \eqref{heat_hardy_psi} taking
\begin{align*}
	h(x,t)=-x^{\alpha}\frac{p(x)}{p(0)}f'(t)\;\;\textrm{ and }\;\; w_0(x)=u_0(x)-x^{\alpha}\frac{p(x)}{p(0)}f(0).
\end{align*}

Moreover, we notice that, by definition, $h\in L^2((0,1)\times (0,T))$ and $w_0\in L^2(0,1)$. This allows us to define in a suitable way the solution of our original problem \eqref{heat_hardy_nhb}. 
\begin{definition}
	$\newline$
	\begin{itemize}
		\item[(i)] We say that $u\in C([0,T];L^2(0,1))\cap L^2(0,T;H^1(0,1))$ is the mild solution of \eqref{heat_hardy_nhb} if $\psi$ defined as in \eqref{cv} is the mild solution of \eqref{heat_hardy_psi}.
		\item[(ii)] We say that $u\in C([0,T];H^1(0,1))\cap H^1(0,T;L^2(0,1))\cap L^2(0,T;\mathcal{D}(\mathcal{A}))$ is the strict solution of \eqref{heat_hardy_nhb} if $\psi$ defined as in \eqref{cv} is the strict solution of \eqref{heat_hardy_psi}.
	\end{itemize}
\end{definition}
\noindent Then, we immediately obtain
\begin{theorem}\label{1d_well-posedness_thm}
	Let us consider $f\in H^1(0,T)$ and $\mu< 1/4$. Then, the following assertions hold:
	\begin{itemize}
		\item[(i)] For all $u_0\in L^2(0,1)$, the non-homogeneous boundary problem \eqref{heat_hardy_nhb} admits a unique mild solution 
		\begin{align*}
			u\in C([0,T];L^2(0,1))\cap L^2(0,T;H^1(0,1)).
		\end{align*}
		\item[(ii)] For all $u_0\in H^1(0,1)$, the non-homogeneous boundary problem \eqref{heat_hardy_nhb} admits a unique strict solution 
		\begin{align*}
			u\in C([0,T];H^1(0,1))\cap H^1(0,T;L^2(0,1))\cap L^2(0,T;\mathcal{D}(\mathcal{A})).
		\end{align*}
	\end{itemize}
\end{theorem}

\section{Eigenvalues and eigenfunctions}\label{spectrum_sec}

This Section is devoted to the analysis of the spectrum of the operator $A_{\mu}$. In particular, we will compute explicitly the eigenvalues and eigenfunctions of $A_{\mu}$ associated to Dirichlet homogeneous boundary conditions, and we will present some fundamental properties and useful estimates. This knowledge of the spectrum will be fundamental for applying the moment method. 

\noindent Let us consider the following eigenvalues problem:
\begin{align*}
	\begin{cases}
		\displaystyle -\phi_k''(x)-\frac{\mu}{x^2}\phi_k(x)=\lambda_k\phi_k(x), &x\in(0,1)
		\\[6pt]
		\phi_k(0)=\phi_k(1)=0.
	\end{cases}
\end{align*}
We know that the general solution of the second order ODE 
\begin{align*}
	-\phi_k''(x)-\frac{\mu}{x^2}\phi_k(x)=\lambda_k\phi_k(x)
\end{align*} 
takes the form
\begin{align*}
	\phi_k(x)=c_1\,x^{\frac{1}{2}}\,J_{\nu}\left(\lambda_k^{\frac{1}{2}}\,x\right)
+c_2\,x^{\frac{1}{2}}\,Y_{\nu}\left(\lambda_k^{\frac{1}{2}}\,x\right),
\end{align*}
with $(c_1,c_2)\neq (0,0)$ and
\begin{align}\label{nu}
	\nu:=\frac{1}{2}\sqrt{1-4\mu},
\end{align}
where $J_{\nu}$ and $Y_{\nu}$ are the Bessel functions of order $\nu$, of first and second kind respectively. 

Since we know (see \cite[Sections 5.3 and 5.4]{lebedev1972special}) that $J_{\nu}(0)=0$ and $Y_{\nu}(0)=-\infty$, the boundary condition $\phi_k(0)=0$ is satisfied choosing $c_2=0$ and $c_1\neq 0$. Without losing generality, we will assume $c_1=1$. Concerning the condition at $x=1$, instead, we have
\begin{align*}
	\phi_k(1)=J_{\nu}\left(\lambda_k^{\frac{1}{2}}\right)=0,
\end{align*}
that holds if we take $\lambda_k:=\jnk^2$, where $\jnk$ are the zeros of $J_{\nu}$. Summarizing, we obtained 
\begin{align*}
	\phi_k(x)=x^{\frac{1}{2}}\,J_{\nu}\left(\jnk\,x\right).
\end{align*}
We remind here that the function $J_{\nu}$ is defined as (see, e.g., \cite[Section 5.3]{lebedev1972special})
\begin{align}\label{bessel_def}
	J_{\nu}(x):=\sum_{m\geq 0}\frac{(-1)^m}{m!\Gamma(m+\nu+1)}\left(\frac{x}{2}\right)^{2m+\nu},
\end{align}
where $\Gamma$ is the Euler Gamma function. Moreover, for $\nu\geq -1$ we know that $J_{\nu}$ has an infinite number of real zeros, all of which are simple with the possible exception of $x=0$, that form a strictly increasing sequence 
\begin{align*}
	0<j_{\nu,1}<j_{\nu,2}<\ldots<\jnk\to +\infty, \;\;\textrm{ as }\; k\to +\infty.
\end{align*}

Furthermore, for any $\nu\geq -1/2$ the Bessel functions $J_{\nu}$ enjoy the following orthogonality property in $[0,1]$ (see \cite[Section 5.14]{lebedev1972special}):
\begin{align}\label{bessel_orth}
	\int_0^1 x J_{\nu}(\jnk\,x)J_{\nu}(j_{\nu,\ell}\,x)\,dx=\frac{\delta_{k,\ell}}{2}\Big(J_{\nu+1}(\jnk)\Big)^2.
\end{align}

Here, $\delta_{k,\ell}$ denotes the Kronecker symbol. Besides, the Bessel functions of the first kind satisfy the recurrence formula (see \cite[Section 3.2]{watson1995treatise}):
\begin{align}\label{bessel_rec}
	J_{\nu+1}(x)= \frac{\nu}{x}J_{\nu}(x)-J_{\nu}'(x).
\end{align}
Using \eqref{bessel_orth} and \eqref{bessel_rec}, for all $k,\ell\geq 1$ we can compute
\begin{align*}
	(\phi_k,\phi_{\ell})_{L^2(0,1)} &= \int_0^1 \phi_k(x)\phi_{\ell}(x)\,dx = \int_0^1 x J_{\nu}(\jnk\,x)J_{\nu}(j_{\nu,\ell}\,x)\,dx
	\\
	&=\frac{\delta_{k,\ell}}{2}\Big(J_{\nu+1}(\jnk)\Big)^2=\frac{\delta_{k,\ell}}{2}\Big(\frac{\nu}{\jnk}J_{\nu}(\jnk)-J_{\nu}'(\jnk)\Big)^2=\frac{\delta_{k,\ell}}{2}\Big(J_{\nu}'(\jnk)\Big)^2,
\end{align*}
and we immediately have 
\begin{align*}
	\norm{\phi_k}{L^2(0,1)}=(\phi_k,\phi_k)_{L^2(0,1)}^{\frac{1}{2}}=\frac{|J_{\nu}'(\jnk)|}{\sqrt{2}}.
\end{align*}

Therefore, we can finally write the normalized eigenfunctions and the spectrum of the operator \eqref{singular_op} on $(0,1)$ with Dirichlet boundary conditions, namely
\begin{align*}
	\Phi_k(x)=\cnk\, x^{\frac{1}{2}}J_{\nu}(\jnk x), \;\;\; \lambda_k=\jnk^2,
\end{align*}
where we introduced the notation $\cnk:=\norm{\phi_k}{L^2(0,1)}^{-1}$. Moreover, it is classical that the family $(\Phi_k)_{k\geq 1}$ forms an orthonormal basis of $L^2(0,1)$.

\subsection{Some bounds for $J_{\nu}$ and its zeros}
Referring to \cite[Section 15.53]{watson1995treatise}, we can give the following asymptotic expansion of the zeros of the Bessel function $J_{\nu}$, for any fixed $\nu\geq 0$:
\begin{align*}
	\jnk=\left(k+\frac{\nu}{2}-\frac{1}{4}\right)\pi-\frac{4\nu^2-1}{8\left(k+\frac{\nu}{2}-\frac{1}{4}\right)\pi}+O\left(\frac{1}{k^3}\right),\;\;\;\textrm{ as } k\to +\infty.
\end{align*}

Moreover, in what follows we will also need the following bounds on the zeros $\jnk$, which are  provided in \cite[Lemma 1]{lorch2008monotonic}
\begin{align}\label{bessel_zero_bound}
	\forall \nu\in \left[0,\frac{1}{2}\right],\;\forall k\geq 1,\;\;\; \pi\left(k+\frac{\nu}{2}-\frac{1}{4}\right)\leq\jnk\leq\pi\left(k+\frac{\nu}{4}-\frac{1}{8}\right), \nonumber
	\\
	\forall \nu\in \left[\frac{1}{2},+\infty\right],\;\forall k\geq 1,\;\;\; \pi\left(k+\frac{\nu}{4}-\frac{1}{8}\right)\leq\jnk\leq\pi\left(k+\frac{\nu}{2}-\frac{1}{4}\right).
\end{align}

The inequalities above become exact when $\nu =1/2$ (corresponding, according to \eqref{nu}, to $\mu=0$). In particular, we have
\begin{align}\label{bessel_zero_upper_bound}
	\jnk\leq k\pi,&\;\;\;\textrm{ for }\; \nu\in\left[0,\frac{1}{2}\right], \nonumber
	\\
	\jnk\leq \left(k+\frac{\nu}{2}\right)\pi, &\;\;\;\textrm{ for }\; \nu\in\left[\frac{1}{2},+\infty\right],
\end{align}
and
\begin{align}\label{bessel_zero_lower_bound}
	\jnk\geq \left(k-\frac{1}{4}\right)\pi,&\;\;\;\textrm{ for }\; \nu\in\left[0,\frac{1}{2}\right], \nonumber
	\\
	\jnk\geq \left(k-\frac{1}{8}\right)\pi, &\;\;\;\textrm{ for }\; \nu\in\left[\frac{1}{2},+\infty\right].
\end{align}

We also recall the following result, whose proof is by now classical and can be found in \cite[Proposition 7.8]{komornik2005fourier}.
\begin{lemma}\label{eigen_gap_lemma}
Let $\jnk$, $k\geq 1$ be the positive zeros of the Bessel function $J_{\nu}$. Then, the following holds:
\begin{enumerate}
	\item The difference sequence $(j_{\nu,k+1}-\jnk)_k$ converges to $\pi$ as $k\to +\infty$.
	
	\item The sequence $(j_{\nu,k+1}-\jnk)_k$ is strictly decreasing if $|\nu| > 1/2$, strictly increasing if $|\nu| < 1/2$, and constant if $|\nu| = 1/2$.
\end{enumerate}
\end{lemma}

In addition, we can easily show that the difference $j_{\nu,k+1}-\jnk$ between two successive eigenvalues is always strictly positive. For $\mu\in[0,1/2]$, this follows employing the estimates \eqref{bessel_zero_upper_bound} and \eqref{bessel_zero_lower_bound}:
\begin{align*}
	\sqrt{\lambda_{k+1}}-\sqrt{\lambda_k} = j_{\nu,k+1}-\jnk\geq \pi\left(k+1-\frac{1}{4}\right)-k\pi = \frac{3}{4}\pi:=\gamma.
\end{align*}

For $\mu\in[1/2,+\infty]$, instead, thanks to Lemma \ref{eigen_gap_lemma} we immediately have that $j_{\nu,k+1}-\jnk>\pi$. Therefore, we can conclude: 
\begin{align}\label{eigen-gap}
	\sqrt{\lambda_{k+1}}-\sqrt{\lambda_k} = j_{\nu,k+1}-\jnk\geq \frac{3}{4}\pi, &\;\;\;\textrm{ for }\; \nu\in\left[0,\frac{1}{2}\right], \nonumber
	\\
	\sqrt{\lambda_{k+1}}-\sqrt{\lambda_k} = j_{\nu,k+1}-\jnk\geq \pi, &\;\;\;\textrm{ for }\; \nu\in\left[\frac{1}{2},+\infty\right]. 
\end{align}

Finally, for our further computations we will need the following bound for the Bessel function $J_{\nu}$, which is presented in \cite{landau2000bessel}:
\begin{align}\label{bessel_landau}
	\forall\,\nu>0,\;\forall\, x>0,\;\;|J_{\nu}(x)|\leq\nu^{-\frac{1}{3}}.
\end{align}

\begin{remark}
We notice that in the case $\mu=0$, corresponding to the classical heat equation without potential, we have $\nu=1/2$ and 
\begin{align*}
	\Phi_k(x)=C_{\frac{1}{2},k}\, x^{\frac{1}{2}}J_{\frac{1}{2}}(j_{\frac{1}{2},k} x).
\end{align*}	
	
Using the definition of Bessel function \eqref{bessel_def}, \eqref{bessel_zero_bound} and classical properties of the Gamma function (see, e.g. \cite[Section 1.2]{lebedev1972special}) we have
\begin{align*}
	J_{\frac{1}{2}}(x) = \frac{\sqrt{2}}{\sqrt{\pi x}}\sin(x),\;\;\;	j_{\frac{1}{2},k} = k\pi, \;\;\; C_{\frac{1}{2},k} = \sqrt{\pi},
\end{align*} 
and the corresponding eigenfunctions becomes $\Phi_k(x)=\sqrt{2} \sin(k\pi x)$. Therefore, we recover exactly the spectrum of the one-dimensional Laplace operator on $(0,1)$ with Dirichlet homogeneous boundary conditions.
\end{remark}

\section{Proof of Theorem \ref{control_thm}}\label{moment_sec}

This Section is devoted to the proof of our main result, Theorem \ref{control_thm}. The proof will be divided in three main steps: 
\begin{itemize}
	\item[•] \textbf{Step 1.} Following the ideas of \cite{cannarsa2015cost,cannarsa2018cost,fattorini1971exact}, we will reduce our control problem to a moment problem of which we will give a formal solution, using the properties of the spectrum of the operator $A_{\mu}$ that we introduced in Section \ref{spectrum_sec}. At this stage, we will also define explicitly the control function $f$ that we shall employ.
	\item[•] \textbf{Step 2.} We prove that, if the condition \eqref{fourier_cond_target} is satisfied, then the control function $f$ is $H^1(0,T)$ regular.
	\item[•] \textbf{Step 3.} We show that the control function $f$ is able to drive the solution $u$ of \eqref{heat_hardy_nhb} from the initial state $u_0$ to the target $u_T$.
\end{itemize}

\subsection{Reduction to a moment problem} 
In this part, we treat the problem with formal computations. We will present a rigorous justification in a second moment. 

Let us start expanding the initial condition $u_0\in L^2(0,1)$ and the target $u_T\in L^2(0,1)$ with respect to the basis of the eigenfunctions $(\Phi_k)_{k\geq 1}$. Indeed, we know that there exist two sequences $(\,\rho_k^0)_{k\geq 1},(\,\rho_k^T)_{k\geq 1}\in \ell^2(\NN^*)$ such that, for all $x\in(0,1)$,
\begin{align*}
	u_0(x)=\sum_{k\geq 1}\rho_k^0\Phi_k(x),\;\;\; u_T(x)=\sum_{k\geq 1}\rho_k^T\Phi_k(x).
\end{align*}
Next, we expand also the solution $u$ of \eqref{heat_hardy_nhb} as
\begin{align*}
	u(x,t)=\sum_{k\geq 1}\beta_k(t)\Phi_k(x), \;\;\; (x,t)\in (0,1)\times(0,T),
\end{align*}
with
\begin{align*}
	\sum_{k\geq 1}\beta_k(t)^2<+\infty.
\end{align*}
Therefore, the controllability condition $u(x,T)=u_T(x)$ becomes 
\begin{align}\label{controllability_moment}
	\forall k\geq 1,\;\;\; \beta_k(T)=\rho_k^T.
\end{align}

On the other hand, we notice that the function $v_k(x,t):=\Phi_k(x)e^{\lambda_k(t-T)}$ solves the adjoint problem
\begin{align}\label{heat_hardy_adj}
	\begin{cases}
		\displaystyle v_{k,t}+v_{k,xx}+\frac{\mu}{x^2}v_k=0 & (x,t)\in Q
		\\[6pt]
		v_k(0,t)=v_k(1,t)=0 & t\in (0,T).
	\end{cases}
\end{align}
Combining \eqref{heat_hardy_nhb} and \eqref{heat_hardy_adj} we obtain
\begin{align*}
	0=& \int_0^1 \bigg[v_k\left(u_t-u_{xx}-\frac{\mu}{x^2}u\right)+u\left(v_{k,t}+v_{k,xx}+\frac{\mu}{x^2}v_k\right)\bigg]\,dxdt
	\\
	=& \int_0^1 v_ku\,\Big|_0^T\,dx - \int_0^T v_ku_x\,\Big|_0^1\,dt + \int_0^T uv_{k,x}\,\Big|_0^1\,dt
	\\
	=& \int_0^1 v_k(x,T)u(x,T)\,dx - \int_0^1 v_k(x,0)u(x,0)\,dx 
	\\
	&+ \int_0^T u(1,t)v_{k,x}(1,t)\,dt - \int_0^T u(0,t)v_{k,x}(0,t)\,dt
	\\
	=& \int_0^1 u(x,T)\Phi_k(x)\,dx - \int_0^1 u_0(x)\Phi_k(x)e^{-\lambda_k T}\,dx - \int_0^T f(t)\left(x^{\alpha}v_{k,x}\right)\Big|_{x=0}\,dt
	\\
	=& \;\beta_k(T)-\rho_k^0e^{-\lambda_k T}-r_k\int_0^T f(t)e^{\lambda_k(t-T)}\,dt,
\end{align*}
where we defined
\begin{align}\label{r_k}
	r_k:=\lim_{x\to 0^+} x^{\alpha}\Phi_k'(x).
\end{align}
It follows that
\begin{align*}
	\forall k\geq 1,\;\;\;\beta_k(T)=\rho_k^0e^{-\lambda_k T}+r_k\int_0^T f(t)e^{\lambda_k(t-T)}\,dt;
\end{align*}
hence, the controllability condition \eqref{controllability_moment} implies
\begin{align}\label{moment_cond}
	\forall k\geq 1,\;\;\;r_k\int_0^T f(t)e^{\lambda_kt}\,dt = -\rho_k^0 + \rho_k^Te^{\lambda_k T}.
\end{align}

On the other hand, since we are looking for a solution of the moment problem belonging to $H^1(0,T)$, instead of \eqref{moment_cond} we would rather be interested in a condition involving the derivative of the function $f$. This condition can be obtained integrating by parts in \eqref{moment_cond}, as follows
\begin{align*}
	\int_0^T f(t)e^{\lambda_kt}\,dt = \frac{1}{\lambda_k}f(t)e^{\lambda_k t}\,\bigg|_0^T-\frac{1}{\lambda_k}\int_0^T f'(t)e^{\lambda_k t}\,dt.
\end{align*}
Therefore, the derivative $f'(t)$ has to satisfy
\begin{align}\label{moment_cond_H1}
	\forall k\geq 1,\;\;\;-\frac{r_k}{\lambda_k}\int_0^T f'(t)e^{\lambda_kt}\,dt = -\rho_k^0 +\rho_k^Te^{\lambda_k T}-\frac{r_k}{\lambda_k}\left(f(T)e^{\lambda_k T}-f(0)\right).
\end{align}

\subsubsection{Computation of $r_k$}
For proving the existence of a function $f(t)$ for which the condition \eqref{moment_cond_H1} holds, it will be necessary to know whether $r_k\neq 0$ for all $k$. This property is guaranteed by the following result.
\begin{lemma}
Let $\alpha$ be defined as in \eqref{alpha}. For the eigenfunction $\Phi_k$ it holds 
\begin{align}\label{eigen_asympt}
	r_k:=\lim_{x\to 0^+} x^{\alpha}\Phi_k'(x)=\frac{\cnk\jnk^{\nu}}{2^{\nu}\Gamma(\nu+1)}\left(\frac{1}{2}+\nu\jnk\right)>0, \;\;\; \forall k\geq 1.
\end{align}
Moreover, we have the following asymptotic behavior 
\begin{align}\label{r_k_lim}
	r_k\sim A_{\nu}\jnk^{\nu+\frac{3}{2}}, \;\;\;\textrm{ as }\;k\to +\infty, \;\;\;\textrm{ with } \; A_{\nu}:=\frac{\nu\sqrt{\pi}}{2^{\nu}\Gamma(\nu+1)}.
\end{align}
\end{lemma}
\begin{proof}
We recall that 
\begin{align*}
	\Phi_k(x)=\cnk\, x^{\frac{1}{2}}J_{\nu}(\jnk x), \;\;\;\textrm{ with }\;\;\; \cnk=\frac{\sqrt{2}}{|J_{\nu}'(\jnk)|}.
\end{align*}
Thus, a direct computation gives 
\begin{align}\label{eigen_deriv}
	x^{\alpha}\Phi_k'(x) = \frac{\cnk}{2}x^{\alpha-\frac{1}{2}}J_{\nu}(\jnk x)+\cnk\jnk x^{\alpha+\frac{1}{2}}J_{\nu}'(\jnk x).
\end{align}
	
Moreover, from the definition of $J_{\nu}$ given in \eqref{bessel_def}, it is straightforward to obtain the following property: for all $\nu\geq 0$
\begin{align}\label{bessel_lim}
	J_{\nu}(x)\sim \frac{1}{\Gamma(\nu+1)}\left(\frac{x}{2}\right)^{\nu},\;\;\;\textrm{ as }\;x\to 0^+.
\end{align}
Using \eqref{bessel_lim} in \eqref{eigen_deriv}, we obtain
\begin{align*}
	r_k =\lim_{x\to 0^+}\left[\frac{\cnk}{2}x^{\alpha-\frac{1}{2}}J_{\nu}(\jnk x)+\cnk\jnk x^{\alpha+\frac{1}{2}}J_{\nu}'(\jnk x)\right] =\lim_{x\to 0^+}\frac{\cnk\jnk^{\nu}}{2^{\nu}\Gamma(\nu+1)}\left(\frac{1}{2}+\nu\jnk\right)x^{\alpha-\frac{1}{2}+\nu}.
\end{align*}
	
Hence, \eqref{eigen_asympt} follows from the definition of $\alpha$. Moreover, we clearly have $r_k>0$ for all $k\geq 1$.
	
For obtaining the behavior of $r_k$ as $k\to +\infty$, we will need to use the following further property of the Bessel function $J_{\nu}$ (see \cite[Section 7.21]{watson1995treatise}):
\begin{align}\label{bessel_asympt}
	J_{\nu}(\xi)^2+J_{\nu+1}(\xi)^2\sim \frac{2}{\pi\xi},\;\;\;\textrm{ as }\;\xi\to +\infty.
\end{align}
In particular, from \eqref{bessel_asympt} we have
\begin{align*}
	J_{\nu+1}(\jnk)^2\sim \frac{2}{\pi\jnk},\;\;\;\textrm{ as }\;k\to +\infty,
\end{align*} 
and this immediately implies that $\cnk\sim (\pi\jnk)^{\frac{1}{2}}$, as $k\to +\infty$. Therefore, 
\begin{align*}
	\lim_{k\to +\infty}r_k &:= \lim_{k\to +\infty}\frac{\cnk\jnk^{\nu+1}}{2^{\nu}\Gamma(\nu+1)}\left(\nu+\frac{1}{2}\jnk^{\,-1}\right)
	\\
	&= \lim_{k\to +\infty}\frac{\sqrt{\pi}\jnk^{\nu+\frac{3}{2}}}{2^{\nu}\Gamma(\nu+1)}\left(\nu+\frac{1}{2}\jnk^{\,-1}\right) = \frac{\nu\sqrt{\pi}}{2^{\nu}\Gamma(\nu+1)}\lim_{k\to +\infty}\jnk^{\nu+\frac{3}{2}},
\end{align*}
and the proof is concluded.
\end{proof}

\subsection{Formal solution of the moment problem}\label{moment_sol_sec}
We present here the formal computations that show that the moment problem \eqref{moment_cond_H1} has a solution $f$. We leave the rigorous justification of these computations, as well as the proof of the $H^1$ regularity of $f$, to the next (sub)section.

For defining the function $f$ satisfying \eqref{moment_cond_H1}, in what follows we firstly need to introduce a sequence $(\sigma_\ell)_{\ell\geq 0}$ in $L^2(0,T)$ which is biorthogonal to $\left(e^{\lambda_k t}\right)_{k\geq 0}$, that is
\begin{align*}
	\int_0^T \sigma_\ell(t)e^{\lambda_k t}\,dt =\delta_{k,\ell}.
\end{align*}

The existence of such a sequence is a consequence of the gap condition \eqref{eigen-gap}, and it is guaranteed by the following result (\cite[Theorem 2.4 and Corollary 1]{cannarsa2015cost}, see also \cite{fattorini1971exact,fattorini1974uniform}). 
\begin{theorem}\label{thm-biortho1-gen}
Assume that $\forall k\geq 0, \lambda_k \geq 0$, and that there is some $\gamma _{\text{min}}>0$ such that
\begin{align*}
	\forall k \geq 0, \sqrt{\lambda_{k+1}} - \sqrt{\lambda_k}  \geq \gamma _{\text{min}} .
\end{align*}
Then there exists a family $(\sigma _\ell)_{\ell\geq 0}$ which is biorthogonal to the family $(e^{\lambda _kt})_{k\geq 0}$ in $L^2(0,T)$:
\begin{align*}
	\forall \ell,k \geq 0, \quad \int _0 ^T \sigma_\ell(t)e^{\lambda_kt} \, dt = \delta_{k,\ell} .
\end{align*}
Moreover, it satisfies
\begin{align}\label{sigma_int}
	\forall k\geq 1,\;\;\; \int_0^T \sigma_k(t)\,dt =0
\end{align} 
and the $L^2(0,T)$-bound
\begin{align}\label{sigma_est}
	\forall k\geq 1,\;\;\; \norm{\sigma_k}{L^2(0,T)}\leq \frac{C(T+1)}{T}e^{P\sqrt{\lambda_k}}e^{-\lambda_k T}e^{\frac{C}{T}};
\end{align}
where $C$ is a universal constant independent of $T$, $\gamma _{\text{min}}$ and $\ell$. 
\end{theorem}

\noindent Now, let us define the function $f$ as follows:
\begin{align}\label{f_def}
	f(t):=\int_0^t g(s)\,ds, \;\;\;\textrm{ with }\;g(t):=\sum_{k\geq 1} \frac{\lambda_k}{r_k}\Big(\rho_k^0-\rho_k^Te^{\lambda_k T}\Big)\sigma_k(t).
\end{align}

It is straightforward that, if $g\in L^2(0,T)$, then $f\in H^1(0,T)$ with $f(0)=0$ and $f'(t)=g(t)$; moreover thanks to \eqref{sigma_int} we have, at least formally,
\begin{align*}
	f(T)=\int_0^T \sum_{k\geq 1} \frac{\lambda_k}{r_k}\Big(\rho_k^0-\rho_k^Te^{\lambda_k T}\Big)\sigma_k(s)\,ds = \sum_{k\geq 1} \frac{\lambda_k}{r_k}\Big(\rho_k^0-\rho_k^Te^{\lambda_k T}\Big)\int_0^T\sigma_k(s)\,ds=0.
\end{align*}
Finally,
\begin{align*}
	-\frac{r_k}{\lambda_k}\int_0^T f'(t)e^{\lambda_k t}\,dt &= -\frac{r_k}{\lambda_k}\int_0^T g(t)e^{\lambda_k t}\,dt = -\frac{r_k}{\lambda_k}\int_0^T \left(\sum_{\ell\geq 1} \frac{\lambda_{\ell}}{r_{\ell}}\Big(\rho_{\ell}^0-\rho_{\ell}^Te^{\lambda_{\ell} T}\Big)\sigma_{\ell}(t)\right)e^{\lambda_k t}\,dt
	\\
	&= -\frac{r_k}{\lambda_k}\sum_{\ell\geq 1}\frac{\lambda_{\ell}}{r_{\ell}}\Big(\rho_{\ell}^0-\rho_{\ell}^Te^{\lambda_{\ell} T}\Big)\int_0^T \sigma_{\ell}(t)e^{\lambda_k t}\,dt 
	\\
	&= -\frac{r_k}{\lambda_k}\sum_{\ell\geq 1}\frac{\lambda_{\ell}}{r_{\ell}}\Big(\rho_{\ell}^0-\rho_{\ell}^Te^{\lambda_{\ell} T}\Big)\delta_{k,\ell} = -\rho_k^0+\rho_k^Te^{\lambda_k T},
\end{align*}
and the moment problem \eqref{moment_cond_H1} is formally satisfied.

\subsection{$H^1$ regularity of the control and controllability result}

\subsubsection{The control $f$ belongs to $H^1(0,T)$}
We have to check that the control $f$ defined as in \eqref{f_def} belongs to $H^1(0,T)$. To this end, we are going to prove, instead, that the function $g$ belongs to $L^2(0,T)$. First of all, we notice that we can split $g:=g^0-g^T$ with
\begin{align*}
	g^0(t):=\sum_{k\geq 1} \frac{\lambda_k}{r_k}\rho_k^0\sigma_k(t)
\end{align*}
and
\begin{align*}
	g^T(t):=\sum_{k\geq 1} \frac{\lambda_k}{r_k}\rho_k^Te^{\lambda_k T}\sigma_k(t).
\end{align*}
Now, from \eqref{r_k_lim} and \eqref{sigma_est} we have 
\begin{align*}
	\norm{g^0(t)}{L^2(0,T)} & \leq C(T)\sum_{k\geq 1} |\,\rho_k^0|\left|\frac{\lambda_k}{r_k}\right|e^{P\sqrt{\lambda_k}}e^{-\lambda_k T} \leq C(T)\sum_{k\geq 1} |\,\rho_k^0|\jnk^{\frac{1}{2}-\nu}e^{P\jnk}e^{-\jnk^2 T} 
	\\
	&\leq C(T)\left(\sum_{k\geq 1} |\,\rho_k^0|^2\right)^{\frac{1}{2}}\left(\sum_{k\geq 1} k^{1-2\nu}e^{2P\jnk}e^{-2\jnk^2 T}\right)^{\frac{1}{2}}
	\\
	& \leq C(T)\,\norm{u_0}{L^2(0,1)}\left(\sum_{k\geq 1} k^{1-2\nu}e^{2P\jnk}e^{-2\pi^2\jnk^2 T}\right)^{\frac{1}{2}} < +\infty, 
\end{align*}
where the last series is convergent due to the presence of the exponential with negative sign.

This can be easily shown employing the bounds \eqref{bessel_zero_upper_bound} and \eqref{bessel_zero_lower_bound} for the zeros of $J_{\nu}$. Indeed, let us assume that $\nu\in [0,1/2]$ (the case $\nu\in[1/2,+\infty]$ is analogous and we leave it to the reader). We have

\begin{align*}
	\sum_{k\geq 1} k^{1-2\nu}e^{2P\jnk}e^{-2\pi^2\jnk^2 T} \leq \sum_{k\geq 1} k^{1-2\nu}e^{2P\pi k}e^{-2\pi^2\left(k-\frac{1}{4}\right)^2 T}.
\end{align*}
Now, an explicit computation gives
\begin{align*}
	\lim_{k\to +\infty} \frac{(k+1)^{1-2\nu}e^{2P\pi(k+1)}e^{-2\pi^2\left(k+1-\frac{1}{4}\right)^2T}}{k^{1-2\nu}e^{2P\pi k}e^{-2\pi^2\left(k-\frac{1}{4}\right)^2T}} = e^{2P\pi}\lim_{k\to +\infty} \left(\frac{k+1}{k}\right)^{1-2\nu}e^{-2\pi^2T(2k+2)}=0.
\end{align*}

This immediately ensures the convergence of the series. For what concerns now the estimate of $g^T(t)$, with similar computations we get
\begin{align*}
	\norm{g^T(t)}{L^2(0,T)} & \leq C(T)\sum_{k\geq 1} |\,\rho_k^T|k^{\frac{1}{2}-\nu}e^{P\pi k}.
\end{align*}
Therefore, if we assume that the series 
\begin{align*}
	\sum_{k\geq 1}|\,\rho_k^T|k^{\frac{1}{2}-\nu}e^{P\pi k} = \sum_{k\geq 1}|\,\rho_k^T|k^{\frac{1}{2}-\frac{1}{2}\sqrt{1-4\mu}}e^{P\pi k}.
\end{align*}
is convergent, we have that also $g^T\in L^2(0,T)$ and we can conclude that the function $g$ belongs to $L^2(0,T)$. Thus, the control $f$ belongs to $H^1(0,T)$.

\subsubsection{The control $f$ drives the solution from $u_0$ to $u_T$}
We show in this (sub)section that the control $f$ that we introduced in \eqref{f_def} is able to drive the solution of \eqref{heat_hardy_nhb} from the initial state $u_0$ to the target $u_T$. With this purpose, let us remind the change of variables 
\begin{align*}
	\psi(x,t):=u(x,t)-x^{\alpha}\frac{p(x)}{p(0)}f(t),\;\;\; p(x):=1-x^{1-2\alpha},
\end{align*}
that transforms our original equation \eqref{heat_hardy_nhb} in
\begin{align*}
	\begin{cases}
		\displaystyle \psi_t-\psi_{xx}-\frac{\mu}{x^2}\psi=-x^{\alpha}\frac{p(x)}{p(0)}f'(t), & (x,t)\in Q
		\\[6pt]
		\psi(0,t)=\psi(1,t)=0, & t\in (0,T)
		\\[6pt]
		\psi(x,0)=u_0(x), & x\in (0,1).
	\end{cases}
\end{align*} 
Now, for a fixed $\varepsilon>0$ we have
\begin{align*}
	\int_{\varepsilon}^T&\int_0^1 -x^{\alpha}\frac{p(x)}{p(0)}g(t)\Phi_k(x)e^{\lambda_k t}\,dxdt 
	\\
	&= \int_{\varepsilon}^T\int_0^1 \left(\psi_t-\psi_{xx}-\frac{\mu}{x^2}\psi\right)\Phi_k(x)e^{\lambda_k t}\,dxdt
	\\
	&= \int_0^1 \psi\Phi_ke^{\lambda_k t}\,\Big|_{\varepsilon}^T\,dx + \int_{\varepsilon}^T \psi\Phi_k'e^{\lambda_k t}\,\Big|_0^1\,dt - \int_{\varepsilon}^T\int_0^1 \psi\left(-\Phi_k''-\frac{\mu}{x^2}\Phi_k-\lambda_k\Phi_k\right)e^{\lambda_k t}\,dxdt
	\\
	&=e^{\lambda_k T}\int_0^1 \psi(x,T)\Phi_k(x)\,dx - e^{\lambda_k \varepsilon}\int_0^1 \psi(x,\varepsilon)\Phi_k(x)\,dx;
\end{align*}
hence, taking the limit for $\varepsilon\to 0^+$ we find
\begin{align*}
	\int_Q -x^{\alpha}\frac{p(x)}{p(0)}g(t)\Phi_k(x)e^{\lambda_k t}\,dxdt = e^{\lambda_k T}\int_0^1 \psi(x,T)\Phi_k(x)\,dx - \rho_k^0.
\end{align*}
From this last identity, it immediately follows
\begin{align*}
	e^{\lambda_k T}\int_0^1 \psi(x,T)\Phi_k(x)\,dx &= \rho_k^0 + \left(\int_0^T g(t)e^{\lambda_k t}\,dt\right)\left(\int_0^1 -x^{\alpha}\frac{p(x)}{p(0)}\Phi_k(x)\,dx\right)
	\\
	&= \rho_k^0 -\frac{\lambda_k}{r_k}\left(-\rho_k^0+\rho_k^Te^{\lambda_k T}\right)\left(\int_0^1 -x^{\alpha}\frac{p(x)}{p(0)}\Phi_k(x)\,dx\right).
\end{align*}
Moreover, 
\begin{align*}
	\int_0^1 & -x^{\alpha}\frac{p(x)}{p(0)}\Phi_k(x)\,dx 
	\\
	&= \frac{1}{\lambda_k}\int_0^1 -x^{\alpha}\frac{p(x)}{p(0)}\lambda_k\Phi_k(x)\,dx = \frac{1}{\lambda_k}\int_0^1 x^{\alpha}\frac{p(x)}{p(0)}\left(\Phi_k''(x)+\frac{\mu}{x^2}\Phi_k(x)\right)\,dx
	\\
	&= \frac{1}{\lambda_k}x^{\alpha}\frac{p(x)}{p(0)}\Phi_k'(x)\,\bigg|_0^1 - \frac{1}{\lambda_k}\int_0^1 \left(x^{\alpha}\frac{p(x)}{p(0)}\right)'\Phi_k'(x)\,dx + \frac{1}{\lambda_k}\int_0^1 x^{\alpha}\frac{p(x)}{p(0)}\frac{\mu}{x^2}\Phi_k(x)\,dx 
	\\
	&=-\frac{r_k}{\lambda_k} - \frac{1}{\lambda_k}\left(x^{\alpha}\frac{p(x)}{p(0)}\right)'\Phi_k(x)\,\bigg|_0^1 + \frac{1}{\lambda_k}\int_0^1 \left[\left(x^{\alpha}\frac{p(x)}{p(0)}\right)''+ \mu x^{\,\lambda-2}\frac{p(x)}{p(0)}\right]\Phi_k(x)\,dx 
	\\
	&=-\frac{r_k}{\lambda_k} + \frac{1}{\lambda_k p(0)}\int_0^1 \Big[\left(x^{\alpha}p(x)\right)''+ \mu x^{\,\lambda-2}p(x)\Big]\Phi_k(x)\,dx = -\frac{r_k}{\lambda_k},
\end{align*}
since from the definition of $p(x)$ it is straightforward to check that
\begin{align*}
	\left(x^{\alpha}p(x)\right)''+ \mu x^{\,\lambda-2}p(x)=0. 
\end{align*} 
Hence, we get
\begin{align*}
	e^{\lambda_k T}\int_0^1 \psi(x,T)\Phi_k(x)\,dx = \rho_k^Te^{\lambda_k T},
\end{align*}
which of course implies
\begin{align*}
	\int_0^1 \psi(x,T)\Phi_k(x)\,dx = \rho_k^T = \int_0^1 u_T(x)\Phi_k(x)\,dx.
\end{align*}
Therefore, we have $\psi(x,T)=u_T(x)$ and, from \eqref{cv}, we can finally conclude that
\begin{align*}
	u(x,T)=\psi(x,T)-x^{\alpha}\frac{p(x)}{p(0)}f(T)=u_T(x),
\end{align*}
since $f(T)=0$.

\section{The cost of null controllability}\label{cost_sec}
In Section \ref{moment_sec}, we obtained the boundary controllability of \eqref{heat_hardy_nhb} assuming $\mu\neq 1/4$. We are going to show now that this restriction, that was coming formally from our previous computations, is actually justified. 

This justification will be provided by an analysis of the cost of null controllability, i.e. the cost for driving an initial datum $u_0$ to zero in time $T$. In particular, we are going to prove that the cost of null controllability blows up as $\mu\to 1/4^-$, thus meaning that null-controllability fails in this critical case.

\begin{proof}[Proof of Theorem \ref{cost_thm}] We will follow the argument presented for the proof of \cite[Theorem 2.2]{cannarsa2015cost}. Moreover, since we are interested in analyzing the cost of controllability as $\mu\to 1/4^-$, in what follows we will assume $\mu\in[0,1/4]$ which, we remind, corresponds to $\nu\in[0,1/2]$.
	
\textbf{Step 1. Upper bound}. First of all, we remind that in \eqref{f_def} we have constructed the following admissible control $f$ that drives the solution $u$ of \eqref{heat_hardy_nhb} to zero in time $T$:
\begin{align*}
	f(t):=\int_0^t g(s)\,ds, \;\;\;\textrm{ with }\;g(t):=\sum_{k\geq 1} \frac{\lambda_k}{r_k}\rho_k^0\sigma_k(t).
\end{align*}
Now, by definition of the controllability cost we have
\begin{align*}
	C^{H^1}(\,\mu,u_0)\leq\norm{f}{H^1(0,T)}.
\end{align*}
	
Therefore, we only have to provide an upper bound for the $H^1$ norm of $f$ which, since $f(0)=0$, is equivalent to bound the $L^2$ norm of $g$. This can be done as follows
\begin{align*}
	\norm{g}{L^2(0,T)} &\leq \sum_{k\geq 1} \frac{\lambda_k}{r_k}|\,\rho_k^0|\,\norm{\sigma_k(t)}{L^2(0,T)} = 2^{\nu}\Gamma(\nu+1)\sum_{k\geq 1} \frac{\jnk^{2-\nu}}{\cnk}\left(\frac{1}{2}+\nu\jnk\right)^{-1}|\,\rho_k^0|\,\norm{\sigma_k(t)}{L^2(0,T)}
	\\
	&= 2^{\nu-\frac{1}{2}}\Gamma(\nu+1)\sum_{k\geq 1} \jnk^{2-\nu}|J_{\nu}'(\jnk)|\left(\frac{1}{2}+\nu\jnk\right)^{-1}|\rho_k^0|\norm{\sigma_k(t)}{L^2(0,T)}
	\\
	&\leq \frac{2^{\nu-\frac{1}{2}}\Gamma(\nu+1)}{\nu}\sum_{k\geq 1} \jnk^{1-\nu}|J_{\nu}'(\jnk)|\,|\,\rho_k^0|\,\norm{\sigma_k(t)}{L^2(0,T)}
	\\
	&\leq \frac{2^{\nu-\frac{1}{2}}\Gamma(\nu+1)}{\nu}\left(\sum_{k\geq 1} |\,\rho_k^0|^2\right)^{\frac{1}{2}}\left(\sum_{k\geq 1} \jnk^{2-2\nu}|J_{\nu}'(\jnk)|^2\norm{\sigma_k(t)}{L^2(0,T)}^2\right)^{\frac{1}{2}}.
\end{align*}
From \eqref{bessel_rec} and \eqref{bessel_landau} we have 
\begin{align*}
	|J_{\nu}'(\jnk)|^2 = |J_{\nu+1}(\jnk)|^2\leq (1+\nu)^{-\frac{2}{3}}\leq 1.
\end{align*}
Therefore, employing also \eqref{sigma_est} combined with \eqref{bessel_zero_bound}, we obtain
\begin{align*}
	\sum_{k\geq 1} \jnk^{2-2\nu}|J_{\nu}'(\jnk)|^2\norm{\sigma_k(t)}{L^2(0,T)}^2 \leq Ce^{\frac{C}{T}}\sum_{k\geq 1} k^{2-2\nu}e^{C\jnk}e^{-2\jnk T}\leq Ce^{\frac{C}{T}}\sum_{k\geq 1} k^{2-2\nu}e^{-\pi^2\left(k-\frac{1}{4}\right)^2T}. 
\end{align*} 
	
Moreover, this last series is clearly convergent due to the presence of the exponential with negative argument. This can be verified with the following explicit computation: 
\begin{align*}
	\lim_{k\to +\infty} \frac{(k+1)^{2-2\nu}e^{-\pi^2\left(k+1-\frac{1}{4}\right)^2T}}{k^{2-2\nu}e^{-\pi^2\left(k-\frac{1}{4}\right)^2T}} = \lim_{k\to +\infty} \left(\frac{k+1}{k}\right)^{2-2\nu}e^{-\pi^2T(2k+2)}=0.
\end{align*}
Hence
\begin{align*}
	\norm{g}{L^2(0,T)} &\leq \frac{C}{\nu}e^{\frac{C}{T}}\Big[\pi^{2-2\nu}\Gamma(\nu+1)2^{\nu-\frac{1}{2}}\Big]\,\norm{u_0}{L^2(0,1)} 
	\\
	&= \frac{Ce^{\frac{C}{T}}}{\sqrt{1-4\mu}}\left[\pi^{\,2-\sqrt{1-4\mu}}\Gamma\left(1+\frac{1}{2}\sqrt{1-4\mu}\right)2^{\frac{1}{2}+\frac{1}{2}\sqrt{1-4\mu}}\right]\norm{u_0}{L^2(0,1)}.
\end{align*}
Finally, we observe that for all $\mu\in[0,1/4]$ we have
\begin{align*}
	\pi^{\,2-\sqrt{1-4\mu}}\leq \pi^2, \;\;\; \Gamma\left(1+\frac{1}{2}\sqrt{1-4\mu}\right)\leq\Gamma(1)=1, \;\;\; 2^{\frac{1}{2}+\frac{1}{2}\sqrt{1-4\mu}}\leq 2.
\end{align*}
Therefore, we can conclude  
\begin{align*}
	\norm{g}{L^2(0,T)} \leq \frac{Ce^{\frac{C}{T}}}{\sqrt{1-4\mu}}\norm{u_0}{L^2(0,1)},
\end{align*}
and this, of course, implies 
\begin{align*}
	C^{H^1}(\,\mu,u_0)\leq \frac{Ce^{\frac{C}{T}}}{\sqrt{1-4\mu}}\norm{u_0}{L^2(0,1)}
\end{align*}
and 
\begin{align*}
	C^{H^1}_{\textrm{bd-ctr}}(\,\mu)\leq \frac{Ce^{\frac{C}{T}}}{\sqrt{1-4\mu}}.
\end{align*}
	
\textbf{Step 2. Lower bound}. For obtaining now the lower bounds for the controllability cost, we will need the following result (see  \cite[Theorem 2.5]{cannarsa2015cost})
	
\begin{theorem}\label{thm-guichal-gen}
Assume that $ \forall k \geq 1, \lambda_k \geq 0$, and that there is some $0 < \gamma _{\text{min}} \leq \gamma _{\text{max}}$ such that
\begin{align}\label{gap-max}
	\gamma _{\text{min}} \leq \sqrt{\lambda _{k+1}} - \sqrt{\lambda _{k}}  \leq \gamma _{\text{max}} .
\end{align}
Then any family $(\sigma_\ell)_{\ell\geq 1}$ which is biorthogonal to the family $(e^{\lambda_kt})_{k\geq 1}$ in $L^2(0,T)$ satisfies:
\begin{align}\label{sigma_est_low}
	\norm{\sigma_k}{L^2(0,T)}\geq \frac{C}{(k+1)!\pi^{2k}T^k(T+1)}e^{-\lambda_k T}e^{\frac{C}{T}},
\end{align}
where $C>0$ is a positive constant independent of $T$ and $\ell$.
\end{theorem}
	
Since we are considering $\mu\in[0,1/4]$ (which corresponds to $\nu\in[0,1/2]$, using \eqref{bessel_zero_upper_bound} and \eqref{bessel_zero_lower_bound}, we see that assumption \eqref{gap-max} is satisfied with
\begin{align*}
	\gamma _{\text{min}} := \frac{3\pi}{4}, \quad \text{ and } \quad \gamma _{\text{max}} := \frac{5\pi}{4}.
\end{align*}
	
Moreover, in what follows, we are going to use the moment condition \eqref{moment_cond}, that in this case reads
\begin{align}\label{moment_cond_zero}
	r_k\int_0^T f(t)e^{\lambda_k t}\,dt=-\rho_k^0.
\end{align}
	
Choose $k\geq 1$, $u_0 = \Phi_k$ and denote $f_k$ an admissible control. From \eqref{moment_cond_zero} we have that the sequence $(-r_kf_k(t))_{k\geq 1}$ is biorthogonal to $\left(e^{\lambda_k t}\right)_{k\geq 1}$ in $L^2(0,T)$. Then, from \eqref{sigma_est_low} we deduce that 
\begin{align*}
	\norm{r_kf_k(t)}{L^2(0,T)}\geq \frac{C}{(k+1)!\pi^kT^k(T+1)}e^{-\lambda_k T}e^{\frac{C}{T}},
\end{align*}
which, of course, implies
\begin{align*}
	\norm{f_k(t)}{L^2(0,T)}\geq \frac{1}{r_k}\frac{C}{(k+1)!\pi^kT^k(T+1)}e^{-\lambda_k T}e^{\frac{C}{T}}.
\end{align*}
Using now the expression \eqref{r_k} for $r_k$, we have
\begin{align*}
	\norm{f_k(t)}{L^2(0,T)}\geq \frac{2^{\nu-\frac{1}{2}}\Gamma(\nu+1)}{\nu\jnk^{\nu}}|J_{\nu}'(\jnk)|\left(\frac{1}{2\nu}+\jnk\right)^{-1}\frac{C}{(k+1)!\pi^kT^k(T+1)}e^{-\lambda_k T}e^{\frac{C}{T}}.
\end{align*}
	
Therefore, choosing e.g. $k=1$, we obtain that there exists a constant $C$, not depending on $\mu$ or $T$, such that  
\begin{align*}
	\norm{f_1(t)}{L^2(0,T)}\geq \frac{C}{\nu}|J_{\nu}'(\jnk)|\frac{e^{\frac{C}{T}}}{T(T+1)}.
\end{align*}
	
Finally, from \cite[Corollary 2]{cannarsa2015cost} we have that there exists a constant $C>0$, not depending on $\mu$, such that
\begin{align*}
	|J_{\nu}'(\jnk)|\geq C,
\end{align*}
and we thus obtain
\begin{align*}
	\norm{f_1(t)}{L^2(0,T)}\geq \frac{C}{\sqrt{1-4\mu}}\frac{e^{\frac{C}{T}}}{T(T+1)}.
\end{align*}
	
\noindent We conclude proving \eqref{control_cost_est}. First of all, we recall that 
\begin{align}\label{rho_zero_exp}
	\rho_k^0 = (u_0,\Phi_k)_{L^2(0,1)} = \int_0^1 u_0(x)\frac{\sqrt{2}}{|J_{\nu}'(\jnk)|}x^{\frac{1}{2}}J_{\nu}(\jnk x)\,dx.
\end{align}
	
Now, we would like to pass to the limit as $\mu\to 1/4^-$ in this last expression. This procedure is justified by a continuity argument described in \cite[Section 7.2]{cannarsa2015cost}, which we are not going to repeat here. Therefore, taking the limit $\mu\to 1/4^-$ in \eqref{rho_zero_exp} we obtain
\begin{align*}
	\lim_{\mu\to 1/4^-}\rho_k^0 = \int_0^1 u_0(x)\frac{\sqrt{2}}{|J_0'(\jk)|}x^{\frac{1}{2}}J_0(\jk x)\,dx=(u_0,\Phi_{0,k})_{L^2(0,1)},
\end{align*}
where we have set 
\begin{align*}
	\Phi_{0,k}(x):=\frac{\sqrt{2}}{|J_0'(\jk)|}x^{\frac{1}{2}}J_0(\jk x).
\end{align*}
	
We stress the fact that, for all $k\geq 1$, $\Phi_{0,k}(x)$ is the eigenfunction associated to $\nu=0$ which, according to \eqref{nu}, is the value that the parameter $\nu$ takes when $\mu=1/4$. 
	
Moreover, we remind that the family $(\Phi_{0,k})_{k\geq 1}$ is an orthonormal basis of $L^2(0,1)$. Therefore we can find a $k$ such that $(u_0,\Phi_{0,k})_{L^2(0,1)}\neq 0$. This implies that there exists a constant $C_0(u_0)$ and $k_0$ such that, for $\mu$ sufficiently close to $1/4$ we have
\begin{align*}
	|\,\rho_{k_0}^0|\geq C_0(u_0).
\end{align*}
Coming back to \eqref{moment_cond_zero}, we then have
\begin{align*}
	\frac{|\,\rho_{k_0}^0|}{r_{k_0}}\leq\left(\int_0^T e^{2\lambda_{k_0} t}\,dt\right)^{\frac{1}{2}}\left(\int_0^T f(t)^2\,dt\right)^{\frac{1}{2}},
\end{align*}
which implies
\begin{align*}
	\frac{C_0(u_0)}{r_{k_0}}\leq\left(\frac{e^{2\lambda_{k_0} T}-1}{2\lambda_{k_0}}\right)^{\frac{1}{2}}\norm{f}{L^2(0,T)}.
\end{align*}
Therefore,
\begin{align*}
	\norm{f}{L^2(0,T)} &\geq \frac{C_0(u_0)}{r_{k_0}}\left(\frac{2\lambda_{k_0}}{e^{2\lambda_{k_0} T}-1}\right)^{\frac{1}{2}} = \frac{C_0(u_0)2^{\nu}\Gamma(\nu+1)}{C_{\nu,k_0}j_{\nu,k_0}^{\nu}}\left(\frac{1}{2}+\nu j_{\nu,k_0}\right)^{-1}\left(\frac{2\lambda_{k_0}}{e^{2\lambda_{k_0} T}-1}\right)^{\frac{1}{2}}
	\\
	&\geq C_0(u_0)\frac{2^{\nu}\Gamma(\nu+1)}{\nu}\left[C_{\nu,k_0}j_{\nu,k_0}^{\nu}\left(\frac{1}{2}+\nu j_{\nu,k_0}\right)\right]^{-1}\left(\frac{2j_{\nu,k_0}}{e^{2j_{\nu,k_0} T}-1}\right)^{\frac{1}{2}}
	\\
	&= C_0(u_0)\frac{2^{\nu-\frac{1}{2}}\Gamma(\nu+1)}{\nu}\frac{|J_{\nu}'(j_{\nu,k_0})|}{j_{\nu,k_0}^{\nu}}\left(\frac{1}{2}+\nu j_{\nu,k_0}\right)^{-1}\left(\frac{2j_{\nu,k_0}}{e^{2j_{\nu,k_0} T}-1}\right)^{\frac{1}{2}}.
\end{align*}
	
Now, we notice that $|J_{\nu}'(j_{\nu,k_0})|$ is bounded from below by a positive constant. On the other hand, also $j_{\nu,k_0}$ is bounded from above and from below by constants depending on $k_0$ (hence, on $u_0$) but uniform with respect to $\mu$. Therefore, there exist another constant $C_1(u_0)$ such that 
\begin{align*}
	\norm{f}{L^2(0,T)} &\geq \frac{C_1(u_0)}{\sqrt{1-4\mu}}\frac{1}{\sqrt{e^{CT}-1}}.
\end{align*}
	
Of course, also the $H^1$ norm of $f$ will satisfy the same estimate, and this concludes our proof.
\end{proof}

\begin{remark}
In the proof of Theorem \ref{cost_thm}, we always assumed $\mu\in[0,1/4]$, which implies $\nu\in[0,1/2]$. We mention that the result can be extended also to negative values of $\mu$, but this has to be done with some more care. Indeed, in this case, according to \eqref{bessel_zero_upper_bound} and \eqref{bessel_zero_lower_bound}, we see that the values $\gamma_{\textrm{min}}$ and $\gamma_{\textrm{max}}$ in \eqref{gap-max} depends also on $\nu$ and, in particular, this gap estimate becomes very bad as $\mu\to -\infty$. In view of that, sharper estimates are needed. They have been provided in \cite{cannarsa2017precise}, and their employment in the analysis of the controllability cost has been presented, e.g., in \cite{cannarsa2015cost}. Since for the problem that we are considering this proof is totally analogous to the one in the aforementioned paper, we leave it to the reader.
\end{remark}

\section{Structure of the targets: proof of Propositons \ref{target_reg_prop} and \ref{reachable_prop}}\label{target_reg_sec}
In this Section, we are going to show that the targets $u_T$ that can be reached from $u_0$ employing an $H^1(0,T)$ control $f$ located at $x=0$ are holomorphic. Once obtained this regularity result, we will also prove that the only target reachable for all values of the coefficient $\mu<1/4$ is $u_T=0$.

\begin{proof}[Proof of Proposition \ref{target_reg_prop}]
Consider the expansion of the target $u_T$ in the basis of the eigenfunctions $\Phi_k$
\begin{align*}
	u_T(x)=\sum_{k\geq 1} \rho_k^T\Phi_k(x)= \sum_{k\geq 1} \rho_k^T\cnk x^{\frac{1}{2}}J_{\nu}(\jnk x).
\end{align*}
Using the definition of the Bessel function $J_{\nu}$ we have
\begin{align*}
	u_T(x)=\sum_{k\geq 1} \rho_k^T\cnk x^{\frac{1}{2}}\left(\sum_{\ell\geq 0} \frac{(-1)^\ell}{\ell!\Gamma(\ell+\nu+1)}\left(\frac{\jnk x}{2}\right)^{2\ell+\nu}\right)=\sum_{k\geq 1} \rho_k^T\cnk x^{\frac{1}{2}}\left(\sum_{\ell\geq 0} d_{\nu,\ell}(\jnk x)^{2\ell+\nu}\right),
\end{align*}
where we defined
\begin{align*}
	d_{\nu,\ell}:= \frac{(-1)^\ell}{\ell!2^{2\ell+\nu}\Gamma(\ell+\nu+1)}.
\end{align*}
Now, formally exchanging the sums we obtain 
\begin{align*}
	u_T(x)=\sum_{\ell\geq 0} d_{\nu,\ell}x^{2\ell+\nu+\frac{1}{2}}\left(\sum_{k\geq 1} \rho_k^T\cnk\jnk^{2\ell+\nu}\right).
\end{align*}
This is precisely in the form
\begin{align*}
	u_T(x)=x^{\nu+\frac{1}{2}}F(x),
\end{align*}
with 
\begin{align*}
	F(x):=\sum_{\ell\geq 0} d_{\nu,\ell}\left(\sum_{k\geq 1} \rho_k^T\cnk\jnk^{2\ell+\nu}\right)x^{2\ell}.
\end{align*}
which is analytic and even.
\end{proof}

Of course, the above argument requires an accurate justification. In particular, we need a rigorous proof of the analyticity of $F$. We will obtain this proof in two steps: firstly, we are going to show that $F$ is holomorphic in the disc $\{z\in\CC,\,|z|<P/\pi\}$. Secondly, we will extend this result in the horizontal strip $\{z\in\CC,\,|\Im z|<P/\pi\}$.

\begin{lemma}\label{holom_disc_lemma}
	If the sequence $(\,\rho_k^Te^{Pk})_{k\geq 1}$ remains bounded, then the function $F$ is holomorphic in the disc $\{z\in\CC,\,|z|<P/\pi\}$.
\end{lemma}

\begin{proof}
First of all, we recall that 
\begin{align*}
	\cnk=\frac{\sqrt{2}}{|J_\nu'(\jnk)|}\sim (\pi\jnk)^{\frac{1}{2}},\;\;\textrm{ as }\;k\to+\infty.
\end{align*}
	
Hence, using the bound \eqref{bessel_zero_upper_bound} for the eigenvalue $\jnk$, we have that there exists a constant $C>0$, independent of $\ell$, such that
\begin{align*}
	|\cnk\jnk^{2\ell+\nu}|\leq C\jnk^{\,2\ell+\nu+\frac{1}{2}}\leq C(\pi k)^{2\ell+1}.
\end{align*}
	
Now, according to \cite[Lemma 5.2]{cannarsa2015cost}, there exists another positive constant, that we will denote again $C$, such that 
\begin{align*}
	\forall\,m\in\NN,\;\;\sum_{k\geq 1} k^{2\ell+1}e^{-Pk}\leq C\frac{(2\ell+1)!}{P^{2\ell+2}}.
\end{align*}
Therefore
\begin{align*}
	\sum_{k\geq 1} |\cnk|\jnk^{2\ell+\nu}e^{-Pk}\leq C\pi^{2\ell+1}\frac{(2\ell+1)!}{P^{2\ell+2}}.
\end{align*}
On the other hand, we notice that 
\begin{align*}
	|d_{\nu,\ell}|\leq\frac{1}{(\ell !)^24^{\ell}}.
\end{align*} 
Moreover, an easy computation shows that the radius of convergence of the series
\begin{align*}
	\sum_{k\geq 1} \frac{(2\ell+1)!}{(\ell !)^24^{\ell}}\frac{\pi^{2\ell+1}}{P^{2\ell+2}}x^{\ell}
\end{align*}
is $P^2/\pi^2$. Indeed,
\begin{align*}
	\lim_{\ell\to +\infty} \left(\frac{(2\ell+3)!}{((\ell+1)!))^24^{\ell+1}}\frac{\pi^{2\ell+3}}{P^{2\ell+4}}\right)\left(\frac{(2\ell+1)!}{(\ell !)^24^{\ell}}\frac{\pi^{2\ell+1}}{P^{2\ell+2}}\right)^{-1} = \lim_{\ell\to +\infty} \frac{(2\ell+3)(2\ell+2)}{4(\ell+1)^2}\frac{\pi^2}{P^2}=\frac{\pi^2}{P^2}.
\end{align*}
	
This means that, if $|x|^2<P^2/\pi^2$, assuming that the sequence $(\rho_k^Te^{Pk})_{k\geq 1}$ remains bounded, the series defining $F$ is convergent and this concludes the proof of the Lemma.
\end{proof}

Lemma \ref{holom_disc_lemma} tells us that the function $F$ is holomorphic in a neighborhood of $x=0$. We are now going to extend this result, proving that $F$ is holomorphic on the whole horizontal strip $\{z\in\CC,\,|\Im z|<P/\pi\}$. Firstly, we note that from the definition \eqref{bessel_def} we have
\begin{align*}
	J_{\nu}(x)=x^{\nu}L_\nu(x),
\end{align*}
where we denoted 
\begin{align*}
	L_{\nu}(x)=\sum_{m\geq 0}\frac{(-1)^m}{m!2^{2m+\nu}\Gamma(m+\nu+1)}\,x^{2m}.
\end{align*}

Moreover, we know that this function $L_\nu$ is holomorphic in $\CC$. Hence, from the expression of $u_T$ we have
\begin{align*}
	u_T(x)=\sum_{k\geq 1} \rho_k^T\cnk x^{\frac{1}{2}}J_{\nu}(\jnk x)=\sum_{k\geq 1} \rho_k^T\cnk\jnk^{\nu} x^{\nu+\frac{1}{2}}L_{\nu}(\jnk x)=x^{\nu+\frac{1}{2}}\tilde{F}(x),
\end{align*}
with 
\begin{align*}
	\tilde{F}(x):=\sum_{k\geq 1} \rho_k^T\cnk\jnk^{\nu} L_{\nu}(\jnk x).
\end{align*}
Furthermore, we have the following.
\begin{lemma}[Lemma 8.3 of \cite{cannarsa2015cost}]\label{holom_strip_lemma}
	If the sequence $(\rho_k^Te^{Pk})_{k\geq 1}$ remains bounded, then the function $\tilde{F}$ is holomorphic in the strip $\{z\in\CC,\,|\Im z|<P/\pi\}$.
\end{lemma}

Lemma \ref{holom_strip_lemma}, of course, concludes the proof of our result. Moreover, we note that due to analyticity reasons $F$ and $\tilde{F}$ coincide.

\begin{proof}[Proof of Proposition \ref{reachable_prop}]
The proof follows the one of \cite[Proposition 2.5(b)]{cannarsa2015cost}. First of all, from Proposition \ref{target_reg_prop} we have that, if $u_T \in\mathcal{P}_{\mu,T}$, then 
\begin{align*}
	u_T(x)=x^{\frac{1}{2}+\frac{1}{2}\sqrt{1-4\mu}}F(x),
\end{align*}
with $F$ holomorphic. Moreover, if $u_T$ is in the intersection of the $\mathcal{P}_{\mu,T}$, this is true for all $\mu<1/4$.
	
Now, if $u_T$ is not zero, let $\kappa_{\mu}$ be the first integer such that the $\kappa_{\mu}$-th derivative $F^{(\kappa_{\mu})}(0)\neq 0$. Then, we have 
\begin{align*}
	u_T(x) \sim \frac{F^{(\kappa_{\mu})}(0)}{\kappa_{\mu}!}x^{\kappa_{\mu}+\frac{1}{2}+\frac{1}{2}\sqrt{1-4\mu}}, \;\;\; \textrm{ as } \; x\to 0^+.
\end{align*}
	
If now we take another value $\mu'<1/4$, then we have the same behavior for $u_T$ close to $x=0$, but this time for the corresponding value $\kappa_{\mu'}$. This, of course, is possible only if the exponents $\kappa_{\mu}$ and $\kappa_{\mu'}$ are the same. Hence, the quantity 
\begin{align*}
	\kappa_{\mu}+\frac{1}{2}+\frac{1}{2}\sqrt{1-4\mu}
\end{align*}
has to remain constant on $(-\infty,1/4)$. Denote $M$ this constant. Then $\kappa_{\mu} = M -\frac{1}{2}-\frac{1}{2}\sqrt{1-4\mu}$, which implies that $\mu \mapsto \kappa_{\mu}$ is continuous.
On the other hand, $\kappa_{\mu}$ is an integer, hence it has to remain constant with respect to $\mu$. Then, the quantity 
\begin{align*}
	\kappa_{\mu}+\frac{1}{2}+\frac{1}{2}\sqrt{1-4\mu}
\end{align*}
has to remain constant, with some uniform $\kappa$ not depending on $\mu$. This is, of course, is a contradiction.
\end{proof}

\section{Open problems}\label{open_pb}

In the present paper we analyzed the boundary controllability for the one-dimensional heat equation
\begin{align*}
u_t-u_{xx}-\frac{\mu}{x^2}u =0, \;\; (x,t)\in (0,1)\times (0,T),
\end{align*}
acting with a control located at the boundary point $x=0$, where the singularity arises. We present here some open problem and perspective related to our work.

\begin{enumerate}
	\item In this article, the controllability of \eqref{heat_hardy_nhb} has been addressed by employing the moment method. It is however natural to wonder whether other techniques, such as a Lebeau-Robbiano strategy or Carleman estimates could apply in this context. 
	
	Since we have an explicit knowledge of the spectrum of the operator \eqref{singular_op}, we believe that a Lebeau-Robbiano approach could be used without significant difficulties for analyzing control properties of \eqref{heat_hardy_nhb}. 
	
	On the other hand, Carleman estimates techniques would certainly be a more delicate issue. We already mentioned that Carleman estimates for heat equations with singular potentials have successfully been employed with control purposes in several works (\cite{cazacu2014controllability,ervedoza2008control,vancostenoble2009hardy,vancostenoble2011improved}). Nevertheless, these results do not extend to our problem, in which we aim to locate the control on the singularity point. A suitable estimate for studying the controllability of \eqref{heat_hardy_nhb} should instead take into account the fact that the normal derivative of the solution of the equation degenerates approaching $x=0$. However, this is not an easy problem. Since we showed that the first derivative of the solution of \eqref{heat_hardy_nhb} behaves as $x^{-\frac{1}{2}+\frac{1}{2}\sqrt{1-4\mu}}$ when $x\to 0^+$, and since we are interested in taking measurements exactly at that point, we have to choose carefully the weight that we shall employ in the Carleman estimate. 
	
	We believe that this weight should be in the form $\sigma(x,t)=\theta(t)p(x)$, with a function $p$ involving $x^{\,2\lambda+1}$ as leading term, with $\lambda=\frac{1}{2}-\frac{1}{2}\sqrt{1-4\mu}$. 	Nevertheless, this choice appears not to be a suitable one for all values of the coefficient $\mu$, since the quantity $2\lambda+1$ becomes negative for $\mu<-3/4$, hence producing a weight $\sigma$ which is not bounded approaching the boundary. On the other hand, to understand which function could allow to obtain the right boundary term in the inequality, without generating singularities, is not an elementary issue.
	
	To obtain a Carleman estimate for \eqref{heat_hardy_nhb} is, therefore, a very fascinating and challenging problem, which would have many applications and extensions. Just to mention one, it would be interesting to analyze whether it is possible to consider  variants of \eqref{heat_hardy_nhb} involving a source term, in order to then address non-linear problems.
	
	\item The problem treated in this article, apart from being interesting by itself,  is also a preliminary step for the analysis of a more general issue, the one of the boundary controllability of the following heat equation 
	\begin{align}\label{1d_heat_2p}
		u_t-u_{xx}-\frac{\mu_1}{x^2}u-\frac{\mu_2}{(1-x)^2}u=0, \;\;\;\; (x,t)\in(0,1)\times(0,T),
	\end{align}
	involving a singular inverse-square potential whose singularities arise all over the boundary of the space domain $(0,1)$. 
	
	First of all, we mention that this problem cannot be treated with the moment method, since in this case with two singular potentials we do not have an explicit knowledge of the spectrum of the operator. We believe that an approach that could be successful would be to derive an appropriate Carleman estimate for the adjoint problem associated to \eqref{1d_heat_2p}. On the other hand, as we were discussing in point 1, this issue is far from being trivial.
	
	\item Related to equation \eqref{1d_heat_2p}, a natural question would be the following:
	can we deduce a controllability result for an equation degenerating at both endpoints like
	\begin{align}\label{heat_2_deg}
		u_t - \Big(x^{\,\alpha_1}(1-x)^{\,\alpha_2}\Big)_x=0,\;\;\; (x,T)\in (0,1)\times (0,T),
	\end{align} 
	with $\alpha_1, \alpha_2\in [0,1)$ and Dirichlet control at the degenerate point?
	
	This problem has already been solved in a case of a distributed control supported in an open subset $\omega\subset (0,1)$ (see, e.g., \cite{martinez2006carleman}). Therefore, the issue of analyzing boundary controllability arises naturally. The problem is delicate however.
	Indeed, also in this case the moment method would fail, since the spectrum of the operator $\Big(x^{\,\alpha_1}(1-x)^{\,\alpha_2}\Big)_x$ on $(0,1)$ with Dirichlet boundary conditions is not explicit. Therefore, the best approach would be, most likely, to prove a Carleman estimate for the adjoint equation associated to \eqref{heat_2_deg}. On the other hand, this problem is definitively not simple and, in our opinion, one should expect difficulties similar to the ones described before for the case of an equation with two singular potentials. 
	Finally, we mention that, even if the two problems \eqref{1d_heat_2p} and \eqref{heat_2_deg} seem to be in some sense related, possibly this is not completely true. Indeed, the presence of two degeneracies (or, analogously, of two singular potentials) makes extremely difficult to identify a change of variables, assuming that it exists, which is able to transform \eqref{heat_2_deg} in \eqref{1d_heat_2p} and vice-versa. Therefore, it is not to be excluded that the two problems are actually hiding difficulties of different nature and, for this reason, have to be studied separately.   
	
	\item The equation \eqref{1d_heat_2p} that we introduced before is a one-dimensional prototype of a more general one, namely
	\begin{align*}
		u_t-\Delta u-\frac{\mu}{\delta^2}u=0, \;\; (x,t)\in\Omega\times (0,T),	
	\end{align*}
	where $\omega\subset\RR^N$ is a bounded and regular domain and $\delta:= \textrm{dist}(x,\partial\Omega)$ is the distance to the boundary function. We already know (see \cite{biccari2016null}) that this equation is null-controllable with an interior control distributed in an open subset $\omega\subset\Omega$. Therefore, the analysis of boundary controllability properties for this model is a natural extension of the results of \cite{biccari2016null} and a very interesting problem.
	
	\item The moment method has also been successfully applied for treating the controllability of one-dimensional systems of coupled parabolic equations (see, e.g., \cite{fernandez2010boundary}). In particular, there the authors present a precise construction of suitable biorthogonal families, which are then applied to prove the observability of the adjoint by means of an Ingham-type approach. To the best of our knowledge, the aforementioned paper has never been extended to the analysis of the controllability of coupled system of parabolic equations with singular potentials. As a first step in this direction, we propose to focus on a system of the type 
	\begin{align}\label{coupled_syst}
		\begin{cases}
			\displaystyle u_t - u_{xx}-\frac{\mu}{x^2}u = Au, & (x,t)\in Q
			\\[6pt]
			u(0,t) = 0, \;\;\; u(1,t) = Bv,	& t\in(0,T)
			\\[6pt]
			u(x,0) = u_0(x), & x\in(0,1),
		\end{cases}
	\end{align}
	with $u=(u_1,u_2)^T$, and where $A$ and $B$ are, respectively, a suitable coupling matrix and a control operator, chosen so that the Kalman rank condition is satisfied. 
	
	Of course, \eqref{coupled_syst} is a very simple toy model, and other kinds of coupled systems could be considered. In particular, it would certainly be interesting to study the case in which the coupling is done in the singular terms.
	
	\item Finally, it would be interesting to study the problem of controllability to the trajectories for heat equations with singular potentials. Of course, for the case of system \eqref{heat_hardy_nhb} that we are considering in this paper, this is a straightforward consequence of Theorem \ref{control_thm}, since we are in a linear setting. Nevertheless, the situation would change if one considered a nonlinear framework. In this case, the techniques developed in this article cannot be applied, not even after a linearization, since the knowledge of the spectrum is not explicit. In view of that, a different approach has to be attempted, the most natural one being the employment of Carleman estimates. Actually, once one has this tool, we believe that  controllability to the trajectories can be obtained relatively easily, following the classical approach of \cite{fernandez2000null} (see also \cite{araruna2015stackelberg,hernandez2018robust}). This would then apply to the controllability of a nonlinear version of \eqref{heat_hardy_nhb}, both from the interior and from the boundary point away from the singularity. On the other hand, if one would study the same control problem acting from the singularity point, the preliminary difficulty of obtaining a Carleman estimate would again appear immediately.
\end{enumerate}

{\appendix
\section{Link with equations with degenerate coefficients}\label{appendix}
As we mentioned in the introduction, the class of equations analyzed in this work can be related to another type of problems, namely evolution PDEs with variable degenerate coefficients. In more detail, we know that there exists an appropriate change of variables that allows to transform our original equation \eqref{heat_hardy_nhb} in the following one:
\begin{align*}
	\phi_t-(\xi^{\,\beta}\phi_{\xi})_{\xi} = 0, \;\;\; (\xi,t)\in (0,\xi_0)\times(0,T),\;\;\; \xi_0:=\left(\frac{2-\beta}{2}\right)^{\frac{2}{2-\beta}}.
\end{align*}
	
For the sake of completeness, we now present this change of variables. First of all, let us introduce a new function $\phi(x,t)$ defined as 
\begin{align*}
	\phi(x,t)=x^{-\frac{\,\beta}{2(2-\beta)}}u(x,t),
\end{align*}
with
\begin{align}\label{beta}
	\beta=\beta(\,\mu):=\frac{2+8\mu-2\sqrt{1-4\mu}}{3+4\mu}.
\end{align}
Starting from \eqref{heat_hardy_nhb}, we get
\begin{align*}
	x^{\frac{\beta}{2(2-\beta)}}\left\{\phi_t-\phi_{xx}-\frac{\beta}{2-\beta}\frac{\phi_x}{x}-\left[\mu+\frac{\beta}{2(2-\beta)}\left(\frac{\beta}{2(2-\beta)}-1\right)\right]\frac{\phi}{x^2}\right\}=0.
\end{align*}
	
\noindent Moreover, it simply a matter of computation to show that, taking $\beta$ as in \eqref{beta} we have 
\begin{align*}
	\mu+\frac{\beta}{2(2-\beta)}\left(\frac{\beta}{2(2-\beta)}-1\right)=0
\end{align*}
and
\begin{align*}
	\frac{\beta}{2(2-\beta)}=\alpha,
\end{align*}
with $\alpha$ as in \eqref{alpha}. Hence, we obtain the equation
\begin{align*}
	\begin{cases}
		\displaystyle \phi_t-\phi_{xx}-\frac{\beta}{2-\beta}\frac{\phi_x}{x}=0, & (x,t)\in Q
		\\
		\phi(0,t)=f(t),\;\;\phi(1,t)=0, & t\in(0,T)
		\\
		\phi(x,0)=x^{-\frac{\beta}{2(2-\beta)}}u(x,0):=\phi_0(x), & x\in(0,1).
	\end{cases}
\end{align*}
Now, let us introduce the new variable 
\begin{align*}
	\xi:=\xi_0 x^{\frac{2}{2-\beta}},\;\;\;\textrm{ with }\;\;\;\xi_0:=\left(\frac{2-\beta}{2}\right)^{\frac{2}{2-\beta}}.
\end{align*}

First of all, we notice that, for $x\in(0,1)$, we have $\xi\in(0,\xi_0)$. Moreover, it is straightforward to check that
\begin{align*}
	\frac{d^2}{dx^2}=\xi^{\,\beta}\frac{d^2}{d\xi^2}+\frac{\beta}{2}\xi^{\,\beta-1}\frac{d}{d\xi}.
\end{align*}
	
\noindent Thus, we finally obtain the following equation with variable degenerate coefficients
\begin{align}\label{deg_pb_xi}
	\begin{cases}
		\displaystyle \phi_t-(\xi^{\,\beta}\phi_{\xi})_{\xi}=0, & (\xi,t)\in (0,\xi_0)\times(0,T)
		\\
		\phi(0,t)=f(t),\;\;\phi(\xi_0,t)=0, & t\in(0,T)
		\\
		\phi(\xi,0)=\left(\frac{\xi}{\xi_0}\right)^{-\frac{\beta}{4}}\phi_0\left(\left(\frac{\xi}{\xi_0}\right)^{\frac{2-\beta}{2}}\right):=\phi_1(\xi), & \xi\in(0,\xi_0).
	\end{cases}
\end{align}
We conclude by observing the two following facts:
\begin{enumerate}
	\item According to \cite{cannarsa2015cost}, the controllability of \eqref{deg_pb_xi} acting from $\xi=0$ can be obtained only for $\beta\in(0,1)$. This, according to \eqref{beta}, corresponds to $\mu\in (0,1/4)$. Therefore, our controllability result Theorem \ref{control_thm} cannot be obtained directly from the results of \cite{cannarsa2015cost} employing the change of variables above presented.
		
	\item For $\mu=1/4$, we have $\beta=1$. In this case, it is well known that the Dirichlet boundary condition $\phi(0,t)=f(t)$ does not makes sense, since it cannot be defined a trace at $\xi=0$ for the solutions to \eqref{deg_pb_xi}. Instead, one has to consider a boundary condition of Neumann type, namely $(\xi^{\,\beta}\phi_{\xi})(0,t)=f(t)$. More details can be found, e.g., in \cite{cannarsa2005null,cannarsa2008carleman,martinez2006carleman,vancostenoble2011improved}.
		
	This justifies the fact that we are not considering the critical case $\mu=1/4$ while analyzing the well-posedness of our original problem \eqref{heat_hardy_nhb}. In this case, it is possible that the problem is well-posed if we impose a different boundary condition at $x=0$, namely the one corresponding to the Neumann boundary condition mentioned above, applying the inverse change of variable to \eqref{deg_pb_xi}. In any case, we will not investigate this fact in the present paper.
\end{enumerate}
}

\section*{Acknowledgements}
The author wishes to acknowledge Piermarco Cannarsa (Universit\`a degli Studi di Roma II - Tor Vergata), for interesting discussions on the topics treated in this paper. A special thanks also to Sylvain Ervedoza (Universit\'e Paul Sabatier (Toulouse)), for his careful revision of the previous version of this manuscript and for his help in detecting and fixing some mistakes in my original work. 

\bibliography{biblio}

\end{document}